\documentclass[11pt,a4wide]{article}
\usepackage{a4wide}
\usepackage{amsmath}
\usepackage{amsfonts}
\usepackage{calc}
\usepackage{amssymb}
\usepackage{amscd}
\usepackage{amsthm}
\usepackage[all]{xy}
\SelectTips{cm}{10}
\usepackage{graphics}
\usepackage{verbatim}
\usepackage{graphicx}
\usepackage{fancyhdr}
\usepackage{ushort}
\usepackage{enumitem}
\usepackage{color}
\usepackage{mathtools}

\usepackage{bm}

\usepackage{url}
\usepackage{rotating}
\usepackage{titlesec}

\newcommand{\comma}{{},\,}
\newcommand{\col}{\!:\:\!}
\newcommand{\loc}{\!\::\!}

\newcommand{\bbN}{\mathbb N}

\newcommand{\bbR}{\mathbb R}

\newcommand{\bbZ}{\mathbb Z}

\newcommand{\mcA}{\mathcal A}
\newcommand{\mcB}{\mathcal B}
\newcommand{\mcC}{\mathcal C}
\newcommand{\mcD}{\mathcal D}
\newcommand{\mcE}{\mathcal E}

\newcommand{\mcG}{\mathcal G}
\newcommand{\mcH}{\mathcal H}
\newcommand{\mcI}{\mathcal I}
\newcommand{\mcJ}{\mathcal J}

\newcommand{\mcM}{\mathcal M}

\newcommand{\mcO}{\mathcal O}

\newcommand{\Ab}{\mathsf{Ab}}

\newcommand{\Ch}{\mathsf{Ch}}

\newcommand{\Diag}{\sSetI}

\newcommand{\simp}{\mathsf{s}}
\newcommand{\Set}{\mathsf{Set}}

\newcommand{\sSet}{\simp\Set}

\newcommand{\Ho}{\mathop\mathrm{Ho}\nolimits}

\newcommand{\map}{\operatorname{map}}

\newcommand{\op}{\mathrm{op}}

\newcommand{\Ra}{\Rightarrow}

\newcommand{\xra}[1]{\xrightarrow{#1}}
\newcommand{\xla}[1]{\xleftarrow{#1}}
\newcommand{\ra}{\rightarrow}
\newcommand{\la}{\leftarrow}

\newcommand{\hra}{\hookrightarrow}
\newcommand{\xlra}[1]{\xrightarrow{\ #1\ }}
\newcommand{\xlla}[1]{\xleftarrow{\ #1\ }}
\newcommand{\lra}{\longrightarrow}
\newcommand{\lla}{\longleftarrow}

\newcommand{\dra}{\xymatrix@1{\ar@{-->}[r]&}}

\newcommand{\cof}[1][]{\mathbin{\:\!\!\xymatrix@1@C=15pt{{}\ar@{ >->}[r]^{#1} & {}}}}
\newcommand{\fib}[1][]{\mathbin{\:\!\!\xymatrix@1@C=15pt{{}\ar@{->>}[r]^{#1} & {}}}}
\newcommand{\embed}[1][]{\mathbin{\:\!\!\xymatrix@1@C=15pt{{}\ar@{c->}[r]^{#1} & {}}}}
\newcommand{\rla}[1][]{\mathbin{\:\!\!\xymatrix@1@C=15pt{{}\ar@<2pt>[r]^{#1} & {}\ar@<2pt>[l]}}}

\newcommand{\pbscale}{.25}
\newcommand{\pboffset}{.5}

\newcommand{\xycorner}[3]{\save [];#2**{}?(\pbscale)="a",[];#1**{}?(\pbscale);"a"**{}?(\pboffset);"a"**\dir{-},[];#3**{}?(\pbscale);"a"**{}?(\pboffset);"a"**\dir{-} \restore}

\newcommand{\pb}[1][{1}]{\xycorner{[r]**{}?(#1)}{[dr]**{}?(#1)}{[d]**{}?(#1)}}
\newcommand{\po}[1][{1}]{\xycorner{[l]**{}?(#1)}{[ul]**{}?(#1)}{[u]**{}?(#1)}}

\newcommand{\coker}{\mathop\mathrm{coker}\nolimits}
\newcommand{\colim}{\mathop\mathrm{colim}}

\newcommand{\ef}{\mathop\mathrm{ef}\nolimits}
\newcommand{\ev}{\mathop\mathrm{ev}\nolimits}

\newcommand{\Hom}{\mathop\mathrm{Hom}\nolimits}

\newcommand{\id}{\mathop\mathrm{id}\nolimits}

\newcommand{\im}{\mathop\mathrm{im}\nolimits}

\newcommand{\sk}{\mathop\mathrm{sk}\nolimits}

\newcommand{\St}{\mathop\mathrm{St}\nolimits}

\newlength{\hlp}

\newcommand{\leftbox}[2]{{}\phantom{#1} \save []+L*+<.5pc>!!<0pt,\the\fontdimen22\textfont2>!L{#1#2} \restore}
\newcommand{\rightbox}[2]{{}\phantom{#2} \save []+R*+<.5pc>!!<0pt,\the\fontdimen22\textfont2>!R{#1#2} \restore}

\entrymodifiers={+!!<0pt,\the\fontdimen22\textfont2>}

\titleformat{\subsection}[runin]{\bfseries}{\thesubsection}{1ex}{}[.]

\newcommand{\heading}[1]{\subsection{#1}}

\makeatletter
\let\c@equation\c@subsection

\makeatother

\theoremstyle{plain}
\newtheorem{thm}[subsection]{Theorem}
\newtheorem{theorem}[subsection]{Theorem}
\newtheorem*{theorem*}{Informal statement}
\newtheorem{corollary}[subsection]{Corollary}

\newtheorem{lemma}[subsection]{Lemma}
\newtheorem{lem}[subsection]{Lemma}
\newtheorem{assumption}[subsection]{Assumption}

\newtheorem{prop}[subsection]{Proposition}
\newtheorem{proposition}[subsection]{Proposition}

\theoremstyle{definition}
\newtheorem{Def}[subsection]{Definition}

\newtheorem{definition}[subsection]{Definition}
\newtheorem*{sdefinition}{Definition}

\newtheorem{remark}[subsection]{Remark}

\newtheorem{ex}[subsection]{Example}
\newtheorem*{example}{Example}

\newcommand{\conn}{\operatorname{conn}}

\DeclareSymbolFont{symbolsC}{U}{pxsyc}{m}{n}
\DeclareMathSymbol{\coloneqq}{\mathrel}{symbolsC}{"42}

\newcommand{\cofr}{\mathrm{cof}}
\newcommand{\fibr}{\mathrm{fib}}
\newcommand{\pt}{\operatorname{pt}}
\newcommand{\dev}{\mathop\mathrm{{ev}}\nolimits}

\newcommand{\sect}[1]{\POS[l]+R*!!<0pt,\the\fontdimen22\textfont2>{\vphantom{|}}="a";[]+L*!!<0pt,\the\fontdimen22\textfont2>{\vphantom{|}} \ar@<-5pt>@/_2pt/"a"_-{#1}}

\DeclareMathAlphabet{\mathsfsl}{OT1}{cmss}{m}{sl}

\newcommand{\OgsSet}{[{\mathcal{O}_G ^\mathsf{op}},{\mathsf{sSet}}]}

\newcommand{\GsSet}{G{-}\sSet}
\newcommand{\Ocat}{\mathcal{O}}

\newcommand{\Pst}[1]{P(#1)}
\newcommand{\Pnewst}{\Pst{\then}}
\newcommand{\Poldst}{\Pst{\then-1}}

\newcommand{\Polder}{P_m}

\newcommand{\Pzerost}{\Pst{0}}
\renewcommand{\Pi}{P_i}

\newcommand{\Pm}{P_{m}}
\newcommand{\Pmmo}{P_{m-1}}
\newcommand{\Yst}[1]{Y(#1)}

\newcommand{\towercompmap}[2]{#1(#2)}

\newcommand{\Engen}{WK(\pi,\then)}
\newcommand{\Kngen}{\overline WK(\pi,\then)}
\newcommand{\Kpin}{K(\pi,\then)}
\newcommand{\Kpinpo}{K(\pi,\then+1)}

\DeclareMathOperator\MCyltemp{{Cyl}}
\newcommand{\MCyl}[1]{\MCyltemp{({#1})}}

\newcommand{\varphist}[1]{\towercompmap{\varphi}{#1}}
\newcommand{\varphin}{\towercompmap{\varphi}{\then}}

\newcommand{\varphim}{\towercompmap{\varphi}{\them}}
\newcommand{\varphimmo}{\towercompmap{\varphi}{\them-1}}
\newcommand{\ellnmo}{\ell_{n-1}}
\newcommand{\elln}{\ell_n}
\newcommand{\ellm}{\towercompmap{\ell}{\them}}

\newcommand{\Knst}{\overline WK(\pi_\then,\then)}
\newcommand{\Enst}{WK(\pi_\then,\then)}
\newcommand{\Lnst}{K(\pi_\then,\then)}
\newcommand{\deltanst}{{\delta_{n*}}}
\newcommand{\face}[1]{\operatorname{face}_{#1}}

\newcommand{\Lm}{{K_m}}

\newcommand{\vertex}[1]{#1}

\newcommand{\them}{m}
\newcommand{\then}{n}

\newcommand{\thedim}{{n}}

\newcommand{\thedimm}{{k}}
\newcommand{\thedimmm}{{j}}
\newcommand{\stdsimp}[1]{\Delta^{#1}}
\newcommand{\horn}[2]{%
\mbox{$\xy
<0pt,-\the\fontdimen22\textfont2>;p+<.1em,0em>:
{\ar@{-}(0,0.1);(3,7)},
{\ar@{-}(3,7);(6,0.1)},
{\ar@{-}(3.2,7);(6.2,0.1)},
{\ar@{-}(3.4,7);(6.4,0.1)}
\endxy\;\!{}^{#1}_{#2}$}}

\newcommand{\Pnew}{{P_\thedim}}
\newcommand{\Pnewind}[1]{P_\thedim^{#1}}
\newcommand{\Pold}{{P_{\thedim-1}}}

\newcommand{\Pind}[1]{P^{#1}}

\newcommand{\pin}{{\pi_\thedim}}

\newcommand{\Kn}{{\overline WK_{\thedim}}}
\newcommand{\En}{{WK_\thedim}}
\newcommand{\Ln}{{K_\thedim}}

\newcommand{\kn}{{k_\thedim}}

\newcommand{\knst}{{k_{\thedim*}}}

\newcommand{\pn}{{p_\thedim}}
\newcommand{\pnst}{{p_{\thedim*}}}

\newcommand{\qn}{{q_\thedim}}

\newcommand{\jnst}{{j_*}}

\def\bo{\partial} 

\def\t{\tau}

\def\s{\sigma}

\def\co{\circ}

\newcommand{\dY}{Y}
\newcommand{\dX}{X}

\newcommand{\dM}{M}

\newcommand{\dpi}{\pi}

\newcommand{\dP}{P}
\newcommand{\dQ}{Q}
\newcommand{\dK}{K}
\newcommand{\dE}{E}
\newcommand{\dC}{{C}}
\newcommand{\dD}{D}

\newcommand{\dZ}{Z}
\newcommand{\dH}{H}

\newcommand{\ccat}{\mathcal C}

\newcommand{\icat}{\mathcal I}
\newcommand{\icatop}{\mathcal I^\op}
\newcommand{\jcat}{\mathcal J}

\newcommand{\OG}{\mathcal{O}_G}

\newcommand{\Z}{\mathbb{Z}}

\newcommand\diff{d}

\newcommand\cobo{\delta}

\newcommand\kkk{{k}}
\newcommand\Redu{\Rightarrow\!\!\!\!\Rightarrow}
\newcommand\lredu{\Leftarrow\!\!\!\!\Leftarrow}
\newcommand{\steq}{\Leftarrow\!\!\!\!\Leftarrow\!\!\!\!\Rightarrow\!\!\!\!\Rightarrow}

\newcommand{\acts}{\mathbin{\:\!\!\xymatrix@1@C=15pt{{} & {\phantom{|}} \ar@(ul,dl)[]_-{}}\!}}

\newcommand{\on}{{o_n}}
\newcommand{\onmo}{{o_{n-1}}}
\newcommand{\deltan}{{\delta_n}}

\DeclareMathOperator\dCone{{Cone}}

\newcommand{\sfP}{\mathsfsl{P}}

\newcommand{\fn}{{f_\thedim}}

\newdir{c}{{}*!/-5pt/@^{(}}
\newdir{d}{{}*!/-5pt/@_{(}}
\newdir{ >}{{}*!/-5pt/@{>}}
\newdir{s}{{}*!/+10pt/@{}}
\newdir{|>}{{}*!/2pt/@{|}*@{>}}

\newcommand{\clm}{\bf \xymatrix@1@C=15pt{{}\ar@{~>}[r] & {}}}

\DeclareSymbolFont{symbolsC}{U}{pxsyc}{m}{n}
\DeclareMathSymbol{\coloneqq}{\mathrel}{symbolsC}{"42}

\newcommand{\cofrst}[1]{#1{}^{\cofr}}
\newcommand{\fibrst}[1]{#1{}^{\fibr}}
\newcommand{\susp}[2]{#1[#2]}
\newcommand{\towercomp}[2]{#1(#2)}

\newcommand{\Tm}{\towercomp{T}{\them}}

\newcommand{\Tmmoc}{\cofrst{\towercomp{T}{\them-1}}}

\newcommand{\Sm}{\towercomp{S}{\them}}

\newcommand{\Smmoc}{\cofrst{\towercomp{S}{\them-1}}}

\newcommand{\trunc}[2]{\operatorname{tr}_{#2} #1}

\newcommand{\An}{\towercomp{A}{\then}}
\newcommand{\Xn}{\towercomp{X}{\then}}

\newcommand{\setone}{G}
\newcommand{\settwo}{H}
\newcommand{\setthree}{D}
\newcommand{\setfour}{E}
\newcommand{\setfive}{F}
\newcommand{\calsetone}{\mathcal{G}}
\newcommand{\calsettwo}{\mathcal{H}}
\newcommand{\calsetthree}{\mathcal{D}}
\newcommand{\calsetfour}{\mathcal{E}}

\newcommand{\elttwo}{h}
\newcommand{\eltthree}{d}
\newcommand{\eltfour}{e}

\newcommand{\caleltone}{\alpha}
\newcommand{\caleltthree}{\delta}
\newcommand{\caleltfour}{\varepsilon}
\newcommand{\connhom}{\Delta}
\newcommand{\actmap}{a}

\newcommand{\sSetDiag}[1]{\sSet{-}#1}
\newcommand{\sSetI}{\sSetDiag\mcI}
\newcommand{\sSetJ}{\sSetDiag\mcJ}
\newcommand{\sSetOG}{\sSetDiag{\mcO_G}}
\newcommand{\ChDiag}[1]{\Ch{-}#1}
\newcommand{\ChI}{\ChDiag\mcI}
\newcommand{\ChJ}{\ChDiag\mcJ}
\newcommand{\AbDiag}[1]{\Ab{-}#1}
\newcommand{\AbI}{\AbDiag\mcI}
\newcommand{\Tow}{\mathsf{Tow}}
\newcommand{\Town}{\mathsf{Tow}_{\leq \then}}
\newcommand{\CI}{\mcC{-}\mcI}

\newcommand{\extend}[1]{\operatorname{ext}_{#1}}
\newcommand{\extn}{\extend{\then}}
\newcommand{\extz}{\extend{0}}
\newcommand{\interval}{I}
\renewcommand{\ef}{\mathrm{ef}}

\begin{document}
\author{M.\ Filakovsk\'{y}\thanks{Research of M.~F.~was supported by the Austrian Science Fund (FWF), Project P31312-N35 and partially supported by the Charles University project
PRIMUS/21/SCI/014.}\and L.\ Vok\v{r}\'{i}nek\thanks{Research of L.~V.~was supported by the Center of Excellence -- Eduard \v{C}ech Institute (project P201/12/G028 of GA~\v{C}R).}}
\title{Computing homotopy classes for diagrams}
\date{\today}
\maketitle
\begin{abstract}
We present an algorithm that, given finite diagrams of simplicial sets $X$, $A$, $Y$, i.e.\ functors $\icatop \to \sSet$, such that $(X,A)$ is a cellular pair, $\dim X \leq 2 \conn Y$ and $\conn Y \geq 1$, computes the set $[X, Y]^A$ of homotopy classes of maps of diagrams $\ell \colon X \to Y$ extending a given $f \colon A \to Y$. For fixed $n = \dim X$, the running time of the algorithm is polynomial. When the stability condition is dropped, the problem is known to be undecidable. Using Elmendorf's theorem, we deduce an algorithm that, given finite simplicial sets $X$, $A$, $Y$ with an action of a finite group $G$, computes the set $[X,Y]^A_G$ of homotopy classes of equivariant maps $\ell \colon X \to Y$ extending a given equivariant map $f \colon A \to Y$ under the stability assumption $\dim X^H \leq 2 \conn Y^H$ and $\conn Y^H \geq 1$, for all subgroups $H\leq G$. Again, for fixed $n = \dim X$, the algorithm runs in polynomial time. We further apply our results to Tverberg-type problem in computational topology: Given a $k$-dimensional simplicial complex $K$, is there a map $K \to \mathbb{R}^{d}$ without $r$-tuple intersection points? In the metastable range of dimensions, $rd \geq (r+1)k +3$, the problem is shown algorithmically decidable in polynomial time when $k$, $d$ and $r$ are fixed.

\end{abstract}

\section{Introduction}


Determining information about the set of homotopy classes of maps $[X,Y]$ between topological spaces $X,Y$ is one of the classical problems of algebraic topology. Indeed the research in particular instances of the problem such as the computation of higher homotopy groups of spheres $[S^n, S^k]$ led to a development of new algebraic methods and tools such as spectral sequences. Further, many problems such as classification of vector bundles (up to isomorphism) and submanifolds (up to cobordism) can be expressed via computation of a homotopy invariant. 

From an algorithmic perspective, the computation of homotopy invariants unlike the computations of homology and (co)homology groups, is a hard problem. In many cases the problem is even computationally undecidable - for example, it is algorithmically undecidable, whether $\pi_1 (Y)$ is nontrivial, even though $Y$ is a finite simplicial complex.
Despite this, there have been several results in the area. Already in the 1950's Brown~\cite{BrownPost} presented an algorithm that computes the set $[X,Y]^A$ of  homotopy classes of maps $X \to Y$ that extend a specific map $A \to X$ under the conditions that $A,X,Y$, are represented as finite simplicial complexes, $f$ is a simplicial map and $Y$ is $1$-connected and has finite homotopy groups. 

Although the results of \cite{BrownPost} lead for example to an algorithm computing the higher homotopy groups of spheres, Brown himself remarked that the algorithms are impractical for computations.


Further progress was achieved in~\cite{post}, where an algorithm was presented that computes $[X,Y]$ for spaces $X,Y$ given as finite simplicial sets satisfying $\dim X \leq 2 \conn Y$ and $\conn Y \geq 1$ (i.e.\ $Y$ is 1-connected). Here $\dim X$ denotes the dimension of $X$ and $\conn Y$ the connectivity of~$Y$. The algorithm was later detailed in \cite{polypost}, where the computational complexity was further discussed.

In the article \cite{aslep}, the authors extended the results from~\cite{post,polypost} to the case when a finite group $G$ acts freely on the spaces $X,Y$. In particular, they have obtained an algorithm, that given spaces $A,X,Y$ as finite simplicial sets with a \emph{free} action of a finite group $G$ and assuming $\dim X \leq 2 \conn Y$ and $\conn Y \geq 1$, computes the set $[X,Y]^A_G$ of equivariant homotopy classes of maps.

We remark that the stability assumptions in the algorithms are tight -- when $Y$ is not 1-connected, the question whether  $[X, Y]^A$ is nonempty is undecidable by the result of Novikov~\cite{Novikov:AlgorithmicUndecidabilityWordProblem-1955} and for 1-connected $Y$ and $\dim X > 2\conn Y +1$, the nonemptiness of $[X, Y]^A$ is undecidable by~\cite{ext-hard}. 
On the other hand, in the stable range $\dim X \leq 2 \conn Y$, the set $[X, Y]^A _G$ has a structure of an Abelian group and it is this structure that is computed in \cite{polypost,aslep}. In the boundary case when $\dim X = 2 \conn Y+1$, the set $[X, Y]^A _G$ no longer has the structure of an Abelian group, however using the  results of \cite{polypost,aslep}, one can algorithmically decide its nonemptiness.

A number of further related results was presented in recent years, for example, in \cite{FilVokri}, we gave an algorithm that computes the set of pointed homotopy classes of maps $[\Sigma X, Y]^*$ under the assumptions that $X,Y$ are finite simplicial sets and $Y$ is 1-connected, which leads to an algorithm that given two maps $f, g \colon X \to Y$ decides whether $f$ and $g$ are homotopic rel $A$.

The motivation behind the results in \cite{post,polypost,aslep} was the \emph{embeddability problem} from computational topology. In short, can one decide whether a $k$-dimensional simplicial complex $K$ is embeddable in $\bbR^d$? In the so-called \emph{metastable range of dimensions}, this problem is equivalent to deciding the nonemptiness of the set $[K^*, S^{d-1}]_{\bbZ_2}$, where $K^*$ denotes the \emph{deleted product} and both $K^*$ and $S^{d-1}$ are endowed with a free action of $\bbZ_2$. Hence the article~\cite{aslep} was able to provide a solution to the embeddability problem in the metastable range.

Our motivation is likewise based on \emph{Tverberg-type} problem -- a generalization of the \emph{embeddability problem}, where one asks wheteher there is an $r$-almost embedding, i.e.\ a map $f\colon K\to \bbR^{d}$ without $r$-tuple intersection points. The Mabillard Wagner-theorem, formulated in~\cite{MabillardWagner:Elim_II_SoCG-2016} (see~\cite{mabillard_tverberg_proof}, \cite{Skopenkov17-1} and \cite{Skopenkov17-2} for full discussion on its proof), states that under the condition $rd \geq (r+1)k +3$, the Tverberg type problem is equivalent to the nonemptiness of the set $[K^r\setminus\Delta^r, S^{d(r-1)-1}]_{\mathbb S_r}$, where the symmetry group $\mathbb S_r$ acts freely on $K^r\setminus\Delta^r$ and non-freely on $S^{d(r-1)-1}$. Can one determine the nonemptiness algorithmically\footnote{As it turns out, yes, see Theorem~\ref{thm:tverberg} and Section~\ref{s:tverberg}.}? To do so would require a generalization of the results of ~\cite{polypost} and \cite{aslep} to situations with nonfree group actions.

Our approach is based on the application of the \emph{Elmendorf's theorem}~\cite{elmendorf}, which, for a finite group $G$, gives a Quillen equivalence between the category $\GsSet$ of $G$-simplicial sets and the category $\OgsSet$ of functors from the so called orbit category $\Ocat_G$. Hence the computation of the homotopy classes of maps between finite diagrams of spaces leads immediately to results in equivariant homotopy theory.




We formulate our main result as follows:

\begin{theorem}\label{thm:main_formulation}
Let $\icat$ be a finite category and $A,X,Y \colon \icat\to \sSet$ be finite diagrams of finite simplicial sets such that  $(X, A)$ is a cellular pair\footnote{We formally introduce cellularity in Section~\ref{s:math_background}. It is a strenghtening of the notion that $A \to X$ is a cofibration.}, $\conn Y \geq 1$ and
\begin{equation}\label{eq:dimconnectivity}\dim X \leq 2 \cdot \conn Y.\tag{*}\end{equation}
Then there exists an algorithm that computes the set $[X, Y]^A$ of homotopy classes of maps\footnote{More precisely, this is the set of morphisms in the homotopy category of $[\icatop, \sSet]$, i.e.\ homotopy classes of maps from $X$ to a fibrant replacement of $Y$ relative to $A$. Details are explained in Section~\ref{s:math_background}.} $\ell \colon X \to Y$ extending a given map $f \colon A \to Y$. For a fixed $\dim X$, the algorithm runs in polynomial time.
\end{theorem}
By computing $[X,Y]^A$ we mean the computation of the isomorphism type of this Abelian group -- the Abelian group structure is described in Theorem~\ref{thm:stable_situation_heaps} and uses the stability condition~\eqref{eq:dimconnectivity}.
\begin{remark}\label{rem:main}
There are multiple ways to interpret the condition \eqref{eq:dimconnectivity} in the statement of Theorem~\ref{thm:main_formulation} as one can define dimension and connectivity in multiple ways. The most simple one is to use maximum dimension and minimum connectivity i.e.
\[\dim X = \max \{ \dim X(i) \mid i \in \mcI \}, \qquad \conn Y = \min \{ \conn Y(i) \mid i \in \mcI \}.\]
One can also treat both the dimension and the connectivity of diagrams $X,Y$ as collections of numbers indexed by the objects of $\mcI$ and  Theorem~\ref{thm:main_formulation} holds true even if the stability condition is interpreted in this pointwise way. As a further remark, it is even possible to define $\dim X(i)$ as the maximal dimension of a cell at $i$, see~Section~\ref{s:diag_spaces}.
\end{remark}

\begin{remark}\label{rem:subcategory}
The assumption $\conn Y \geq 1$ can be easily removed. The exact algorithm as described in this paper works also for $\conn Y \geq 0$ since it is based on the Postnikov tower of $\Sigma Y$ which is 1-connected in this case. As a matter of fact, it works for $\conn Y \geq -1$, i.e.\ if all the spaces in the diagram are non-empty. With a careful treatment of nonconnected spaces it works even in a complete generality, but there is a simpler solution: let $\mcJ \subseteq \mcI$ denote the full subcategory on objects where $Y$ is nonempty. Then $[X, Y]^A = \emptyset$ unless $X$ is also empty on $\mcI \smallsetminus \mcJ$, in which case $[X, Y]^A = [X|_\mcJ, Y|_\mcJ]^{A|_\mcJ}$ and we are in the simpler situation $\conn Y \geq -1$ as above.
\end{remark}

Let $G$ be a finite group. The orbit category $\Ocat_G$ of $G$ is a finite category with objects $G/H$, for subgroups $H \leq G$, and arrows the equivariant maps $G/H \to G/K$.
Suppose now that $A,X,Y$ are finite simplicial sets with an action of a finite group $G$. By Elemendorf's theorem \cite{elmendorf} the categories $\GsSet$ of $G$-simplicial sets and  $\OgsSet$ are Quillen equivalent, namely $[X,Y]^A _G \cong [\Phi X, \Phi Y]^{\Phi A}$, where 
\[
\Phi (X) (G/H) = X^H = \{x\in X \mid hx = x, \forall h\in H\}.
\]
For details see e.g.\ Chapter V in \cite{may1996equivariant} or  \cite{elmendorf}.

Therefore, as a straightforward consequence of Theorem~\ref{thm:main_formulation} we obtain
\begin{theorem}\label{cor:equiv_main}
Let $A\subseteq X$ and $Y$ be finite simplicial sets with an action of a finite group $G$. Supposing that $\dim \Phi X \leq 2\conn \Phi Y$ and $\conn \Phi Y \geq 1$, there is an algorithm that computes the set $[X,Y]^A _G$ of equivariant homotopy classes of maps $X \to Y$ extending a given equivariant map $A \to Y$. If $\dim X$ is fixed, the algorithm runs in polynomial time.
\end{theorem}

\subsection{Applications}
We present two applications of Theorem~\ref{cor:equiv_main} for computations in equivariant homotopy theory. First describes a solution to the aforementioned Tverberg-type problem from computational topology. The second application is of more theoretical nature and gives an algorithm that computes equivariant stable homotopy groups of spheres. We introduce both here, while proofs of these results are postponed until Section~\ref{sec:applications}.

\subsection*{Tverberg problem}
Let $K$ be a $\thedimm$-dimensional simplicial complex and let $f\colon K \to  \mathbb{R}^d$ be a map. A point $x \in f(K)$ is called an $r$-Tverberg point if it has preimages lying in $r$ pairwise disjoint simplices of $K$.

The \emph{Tverberg-type problem} is a question whether there exists an, \emph{almost $r$-embedding}, i.e.\ a  map $f\colon K \to \mathbb{R}^d$ such that it contains no $r$-Tverberg points. We will show that Theorem~\ref{thm:main_formulation} implies the following:
\begin{theorem}\label{thm:tverberg}
Let $K$ be a $\thedimm$-dimensional simplicial complex, $d,r \in \mathbb{N}$, such that $rd \geq (r+1)k +3$. Then there is a polynomial time algorithm that decides whether there is an almost $r$-embedding $f\colon K \to \bbR^d$.
\end{theorem}

\subsection*{Stable homotopy groups}
As a second application, we obtain 
\begin{thm}\label{thm:equivstable1}
Let $X,Y$ be finite simplicial sets with an action of a finite group $G$. Then there is an algorithm that computes the set $\{X,Y\}_G$ of equivariant stable homotopy classes of maps $X \to Y$. 
\end{thm}

\section{The idea of the proof of the main theorem} \label{s:idea_proof}

The proof consists of three main steps.
\begin{itemize}
\item
Using stability, replace $[X,Y]^A$ by the isomorphic $[\Sigma X, \Sigma Y]^{\Sigma A}$ where $\Sigma$ denotes the unreduced suspension.
\item
Construct the Postnikov tower for $\Sigma Y$, with stages $\Pnewst$ that capture the homotopical information up to dimension $\then$, such that the difference between $\Pnewst$ and $\Poldst$ is isolated in dimension $\then$ and is rather easy to handle explicitly, using a certain ``exact sequence''.
\item
If non-empty, $[\Sigma X, \Pnewst]^{\Sigma A}$ becomes an Abelian group via homotopy concatenation (to be more precise, it is an Abelian heap). Inductively with respect to $n$,  propagate the structure of an ``effective'' Abelian group from $[\Sigma X, \Poldst]^{\Sigma A}$ to $[\Sigma X, \Pnewst]^{\Sigma A}$, using the exact sequence from the previous point.
\end{itemize}
An important part of the mentioned effective Abelian group structure is an algorithm that outputs its isomorphism type. For $n = \dim \Sigma X$, we have $[X,Y]^A \cong [\Sigma X, \Sigma Y]^{\Sigma A} \cong [\Sigma X,\Pnewst]^{\Sigma A}$ and the computation is finished.

\subsection{New ingredients}

Although the structure of the computation is mostly identical to that of our previous work \cite{polypost,aslep}, there are two important differences, the reader should be aware of:

Firstly, the construction of the Postnikov tower requires the Postnikov stages to be cofibrant (more careful explanation is given later in Sections~\ref{s:Postnikov_tower_motivation} and~\ref{s:Postnikov_tower_diagrams}) and this is generally not the case for diagrams. Our solution consists of employing a cofibrant replacement that, however, makes the basic shape of the tower more complicated (we need both the stages and their cofibrant replacements) and renders the Postnikov stages non-fibrant. For this reason, homotopy classes in $[\Sigma X, \Pnewst]^{\Sigma A}$ are not represented by actual maps $\Sigma X \to \Pnewst$. Instead, as a way around these technical problems, we develop a convenient category of \emph{towers}, in which the homotopy classes admit representatives, and use this framework throughout the paper.

Secondly, the exact sequence relating homotopy classes of maps into consecutive Postnikov stages consists more naturally of \emph{unpointed} sets and we follow this more conceptual approach in the paper (it already appeared in \cite{FilVokri}). The main point is that the action of $\Lnst$ on $\Pnewst$, which is free with $\Poldst$ as the space of orbits, induces an action of $[\Sigma X, \Lnst]^{\Sigma A}$ on $[\Sigma X, \Pnewst]^{\Sigma A}$, but with possibly non-trivial stabilizers and with the set of orbits possibly a proper subset of $[\Sigma X, \Poldst]^{\Sigma A}$. The exact sequence captures both the stabilizers and the subset. For details, see Section~\ref{sec:exact_sequence}.

\subsection{Plan of the paper}

We start by setting up the mathematical notions required in the paper, Section~\ref{s:math_background}. Then, in Section~\ref{s:Algorithmic_structures} we explain various ways of endowing these mathematical objects with a computational layer. After that we give a more detailed idea of the proof with more precise statements (Section~\ref{s:main_theorem}), while some more technical aspects of the proof  - the obstruction theory for diagrams, the exact sequence, further effective homological algebra and construction of Postnikov towers are presented in Sections~6--9. Finally, in Section~10, we discuss applications of Theorem~\ref{thm:main_formulation}.


\section{Mathematical background}\label{s:math_background}
\subsection{Model category formalism}

One of the successful formalisms in homotopy theory is that of model categories. We will be dealing with the model category of spaces, $G$-spaces, chain complexes and diagrams in these categories. A precise definition of a model structure will not be needed\footnote{We refer the reader to standard resources \cite{hovey2007model} for full details.} as only a fragment of the model structure in the above examples will be required -- the class of \emph{cofibrant objects}, the class of \emph{fibrant objects},  and a \emph{homotopy relation}. These are crucial concepts since, for $X$ cofibrant and $Y$ fibrant in $\mcC$, the hom-set $[X,Y]$ in the homotopy category $\Ho(\mcC)$ is defined to be the set of homotopy classes of maps from $X$ to $Y$. For general $X$ and $Y$, one needs to choose a cofibrant replacement $X^\cofr$ and a fibrant replacement $Y^\fibr$ and then defines the hom-set in $\Ho(\mcC)$ to be
\[[X,Y] = \map(X^\cofr, Y^\fibr)/{\sim}\]
 We remark that the notion of ``replacement'' also requires the specification of weak equivalences.

The cofibrant objects are usually described via a generating set of cofibrations $K_j \to L_j$, thought of as boundary inclusions of cells (of various shapes). We then say that the pushout $X$ in
\[\xymatrix{
K_j \ar[r] \ar[d] & A \ar[d] \\
L_j \ar[r] & \leftbox{X}{{}=A \cup_{K_j} L_j} \po
}\]
is obtained from $A$ by attaching a cell along the \emph{attaching map} $K_j \to A$, where by an actual \emph{cell} we understand the canonical map $L_j \to X$. Any object $X$ obtained from the initial object $\emptyset$ by successively attaching cells is said to be \emph{cellular} or a \emph{cell complex}; any such $X$ is cofibrant. We remark that the cells are attached in some order and, thus, a cell complex is generally not specified by the collection of cells. For our model categories, it will always be possible to attach cells in the order of increasing dimension and this technical issue disappears.

\subsection{Relative categories}

When $\mcC$ is a model category and $A \in \mcC$ an object, the \emph{slice category} $A/\mcC$, or the category of \emph{objects under} $A$, has as objects maps in $\mcC$ with domain $A$; its maps from $f \colon A \to X$ to $g \colon A \to Y$ are commutative triangles
\[\xymatrix@R=0pt{
& X \ar[dd] \\
A \ar[ru]^-f \ar[rd]_-g \\
& Y
}\]
We will now explain the important model category concepts for $A/\mcC$ in terms of $\mcC$.

The cofibrant objects are cofibrations in the model structure of $\mcC$, while fibrant objects are maps with fibrant codomain.

If $X$ is obtained (in $\mcC$) from $A$ by successively attaching cells, the canonical map $A \to X$ is said to be a \emph{relative cell complex} and these constitute exactly the cell complexes in $A/\mcC$. In our examples, $A$ will always be a subobject of $X$ and we will denote the relative cell complex as a pair $(X, A)$.

The hom-set in $\Ho(A/\mcC)$ will be denoted by $[X, Y]^A$, where we suppress from the notation the involved maps $A \to X$ and $A \to Y$; these will always be fixed and clear from the context. For $(X, A)$ cofibrant and $(Y, A)$ fibrant, this is the set of homotopy classes relative to $A$.

\subsection{Spaces = simplicial sets}

For computational purposes, a space will mean a simplicial set. We denote by $\sSet$ the category of simplicial sets and simplicial maps between them.

We equip simplicial sets with the Kan model structure: Generating cofibrations are the boundary inclusions $\partial\Delta^n \to \Delta^n$, for $\then \geq 0$, where $\Delta^n$ denotes the standard $n$-simplex and $\partial\Delta^n$ the union of all its proper faces. In this way, any simplicial set is cofibrant, in fact cellular. Thus, cells are maps $\Delta^n \to Y$ or equivalently $n$-simplices of $Y$, for arbitrary $\then \geq 0$. The canonical cellular structure on $Y$, unique up to the order of cells, has as cells precisely all non-degenerate simplices.

Fibrant objects, the so called Kan complexes, are simplicial sets that have the right lifting property with respect to the horn inclusions $\horn nk \to \Delta^n$, where $\horn nk$ is the union of all proper faces of $\Delta^n$ with the exception of the $k$-th face. Most importantly for us, all simplicial groups are fibrant.

We will also use the notation $I = \Delta^1$ for the interval, especially when talking about homotopies.

We will denote by $s_J = s_{j_r} \cdots s_{j_1}$ a degeneracy operator for a set $J = \{j_r > \cdots > j_1\}$. For each simplex $x$ there is a unique non-degenerate simplex $\overline x$ and a unique degeneracy $s_J$ such that $x = s_J \overline x$. The set $J$ consists of all the $j$ for which $x$ lies in the image of $s_j$. By definition, $x$ is non-degenerate if and only if $J = \emptyset$.

\subsection{Diagrams}

Let $\mcI$ be a small category. For a category $\mcC$, we will denote by $\CI$ the category of diagrams $\mcI^\op \to \mcC$. We thus have the category $\sSetI$ of diagrams of spaces, $\AbI$ of diagrams of Abelian groups, $\ChI$ of diagrams of chain complexes, etc.\ (as the notation suggests, we think of them as right $\mcI$-modules with values in $\mcC$).

\subsection{Diagrams of spaces}\label{s:diag_spaces}

In particular, we have the category $\Diag$ of $\mcI$-shaped diagrams of spaces. The representable functor $\mcI(-, i) \colon \mcI^\op \to \sSet$ can be interpreted as a functor with values in (discrete) simplicial sets and we will thus write $\mcI(-, i) \in \Diag$.

The model structure on $\Diag$, the so called \emph{projective model structure} which we are about to describe, is more complicated than that on $\sSet$ in that not every object is cofibrant; on the other hand, fibrant objects are simply diagrams consisting of fibrant objects. The generating cofibrations are the maps
\[\partial\Delta^n \times \mcI(-, i) \to \Delta^n \times \mcI(-, i)\]
In this way, a cell is a map $\Delta^n \times \mcI(-, i) \to Y$ or, equivalently, an $n$-simplex of $Y(i)$, for $n \geq 0$ and $i \in \mcI$ arbitrary. This results in the following characterization:

\begin{proposition} \label{prop:cellular_diagrams_spaces}
A diagram $X$ is cellular if and only if there is a collection of simplices $e_\alpha \in (X(i_\alpha))_{n_\alpha}$, for $\alpha \in \mcA$, called cells, such that any simplex $e \in (X(i))_n$ is obtained uniquely from a cell by applying a map in the diagram and a degeneracy, i.e.\ $e = s_J(f^*(e_\alpha))$ for unique $\alpha \in \mcA$, $f \colon i \to i_\alpha$ and degeneracy $s_J$.

More generally, a cellular pair $(X, A)$ is one for which the above condition is satisfied for simplices $e \in X \smallsetminus A$.
\end{proposition}

\subsection{Equivariant spaces}

Let $G$ be a fixed finite group. If we interpret $G$ as a one-object category, spaces with a $G$-action ($G$-spaces) are functors $G \to \sSet$ and their category will be denoted $\GsSet$. This category is equipped with a model structure that is described below and is different from the projective model structure on diagrams of spaces. However, a theorem of Elmendorff says that this category if Quillen equivalent to $\sSetOG$ for the so called category of orbits $\OG$ (consisting of all orbits $G/H$ and all equivariant maps between them). This is how questions of homotopical nature regarding $\GsSet$ are answered: by translating to $\sSetOG$ and solving there.

For the purpose of the translation, it will be useful to describe the generating set of cofibrations for $\GsSet$. They are given by inclusions $\partial \stdsimp \then \times G/H \to \stdsimp \then \times G/H$, for all $\then \geq 0$ and for all subgroups $H$ of $G$. Thus, a cell of $X$ is a map $\stdsimp \then \times G/H \to X$, i.e.\ an $\then$-simplex of the fixed point space $X^H$. Similarly to simplicial sets, every object is cofibrant. The functor $\Phi \colon \GsSet \to \sSetOG$ takes a $G$-space to the collection of its $H$-fixed point subspaces, for all subgroups $H$ of $G$ and all action maps between them. It is not too difficult to see\footnote{%
	However, we remark that this requires commutation of certain limits and colimits and as such does not hold in arbitrary categories, but only in categories exhibiting this kind of ``exactness''.
} that a cell $\stdsimp\then \times G/H \ra X$ gives a cell $\stdsimp \then \times \OG(-, G/H) \to \Phi(X)$ and, in this way, the diagram $\Phi(X)$ is cellular for any cellular $G$-space $X$ (with cells of $\Phi(X)$ corresponding to those of $X$).

\subsection{Chain complexes}

We will be working exclusively with non-negatively graded chain complexes of Abelian groups in their projective model structure, denoted $\Ch$. The free chain complex $D^\then$ generated by $x$ in dimension $\then$ has, for $\then >0$, the Abelian group $\bbZ$ in dimensions $\then$ and $\then - 1$, generated by $x$ and $\partial x$ respectively. Its boundary $\partial D^\then$ is the subcomplex generated by $\partial x$, i.e.\ has $\bbZ$ in dimension $\then - 1$. The case $\then = 0$ is special in that $\partial x = 0$, and thus $D^\then = \bbZ$, $\partial D^\then = 0$. The boundary inclusions $\partial D^\then \to D^\then$ are the generating cofibrations for the projective model structure on chain complexes. Therefore, cells are maps $D^\then \to C$ and correspond to $n$-chains $c \in C_\then$. A cellular chain complex consists of free $\bbZ$-modules with a basis in each dimension formed by the cells viewed as chains.

\subsection{Diagrams of chain complexes}

In the category $\ChI$ of diagrams of chain complexes, cofibrations are generated similarly by boundary inclusions $\partial D^\then_i \to D^\then_i$ where $D^\then_i$ is a free diagram generated by a single element $x$ at object $i$ sitting in dimension $\then$. This has $\bbZ\mcI(-, i)\in \AbI$ (the free Abelian group on a representable diagram) in dimensions $\then$ and $\then - 1$ with boundary the identity. A diagram of chain complexes is cellular if and only if, in each dimension, it is a direct sum of diagrams of the form $\bbZ\mcI(-, i)$. More concretely, we have the following characterization:

\begin{proposition} \label{prop:cellular_diagrams_complexes}
A diagram $C$ is cellular if and only if there is a collection of cells (i.e.\ chains) $c_\alpha \in C(i_\alpha)_{n_\alpha}$, for $\alpha \in \mcA$, such that any chain $c$ can be obtained uniquely from cells by applying maps in the diagram and linear combinations, i.e.
\begin{equation}\label{eq:cellular}
c = \sum_{\alpha \in \mcA, f \colon i \to i_\alpha} k_{\alpha, f} f^*(c_\alpha)
\end{equation}
for unique $k_{\alpha, f} \in \bbZ$ (only a finite number of non-zero coefficients).

More generally, a cellular pair $(C, C')$ is one for which the above condition is satisfied modulo $C'$.
\end{proposition}

As an important example, if $X$ is a diagram of spaces then the \emph{normalized}\footnote{%
	The normalized chain complex of a simplicial set $K$ has $C_n(K)$ freely generated by simplices of $K$ with all degenerate simplices quotiented out.%
} chain complexes $C_*(X(i))$ of the spaces in the diagram form a diagram $C_*(X)$ of chain complexes. Since we are dealing (exclusively) with normalized chain complexes, for a cellular diagram $X$, the diagram $C_*(X)$ of chain complexes is also cellular with cells corresponding bijectively to those of $X$. There is an obvious generalization to the relative situation of a cellular pair $(X, A)$.

\subsection{Bredon cohomology}

For a cellular pair $(X, A)$ of diagrams and for a digram $\pi \in \AbI$ of Abelian groups the cochain complex
\[C^*(X, A; \pi) = \Hom_{\AbI}(C_*(X, A), \pi),\]
equipped with the differential\footnote{%
	Any other choice works equally well, e.g.\ the usual $\delta c = -(-1)^{|c|} \cdot \partial^* c$, as long as this is reflected in the (unspecified) isomorphism $\overline W \Kpin \cong \Kpinpo$ below.
} $\delta = \partial^*$, is called the Bredon cochain complex. The cohomology of this cochain complex $C^*(X, A; \pi)$ is known as \emph{Bredon cohomology}, see \cite{may1996equivariant,bredon1967}. As a functor of the pair $(X, A)$ we will explain shortly that this is represented by an Eilenberg-MacLane diagram.

\subsection{Cofibrant replacement in $\Diag$}

We will use a concrete model for the cofibrant replacement, namely the Bousfield--Kan model. Let $X$ be any diagram. Then the cofibrant replacement $X^\cofr = |BX|$ is a geometric realization of a certain simplicial object $BX$; we start with decribing the involved simplicial object $BX \colon \Delta^\op \to \Diag$; in dimension $\then$ it is
\[(BX)_\then = \coprod_{i_0, \ldots, i_n \in \mcI} Xi_0 \times \mcI(i_1, i_0) \times \cdots \times \mcI(i_\then, i_{\then - 1}) \times \mcI(-, i_\then)\]
with the face map $d_j$ given either by composition, for $j > 0$, with $d_0$ being the right action $Xi_0 \times \mcI(i_1, i_0) \to Xi_1$ coming from $X$ being a contravariant functor $X \colon \mcI^\op \to \sSet$, and with degeneracy maps inserting the identity at various positions.

The geometric realization of $BX$ is then the quotient
\[X^\cofr = \coprod_{n \geq 0;\, i_0, \ldots, i_n \in \mcI} \Delta^\then \times Xi_0 \times \mcI(i_1, i_0) \times \cdots \times \mcI(i_\then, i_{\then - 1}) \times \mcI(-, i_\then) / {\sim},\]
where the relation is similar to that of a tensor product (formally, such a construction is called the coend $\Delta^\bullet *_{\Delta^\op} BX$): we require $(\theta^* t, z) \sim (t, \theta_* z)$, for $t \in \Delta^\them$, $z \in (BX)_\then$ and $\theta$ a morphism in $\Delta^\op$; of course, faces and degeneracies are sufficient to generate all relations.

\begin{lemma} \label{l:cofibrant_replacement_tensors}
For any space $K$ we have $K \times X^\cofr \cong (K \times X)^\cofr$, i.e.\ the cofibrant replacement commutes with $\sSet$-tensors.\qed
\end{lemma}

The cofibrant replacement $X^\cofr$ of any diagram is a cellular diagram. Precise details will not be important for the paper, but are necessary for an implementation of our algorithm. The cells are $(t, x, f_0, \ldots, f_{\then - 1}, \id)$ for any chain
\[i_0 \xlla{f_0} i_1 \lla \cdots \lla i_{\then - 1} \xlla{f_{\then - 1}} i_\then\]
of non-identity morphisms and any non-degenerate simplex $(t, x) \in \Delta^n \times Xi_n$ not contained in $\partial\Delta^n \times Xi_n$. 
The non-degenerate simplices of a product can be described equivalently as pairs $(s_J \overline t, s_K \overline x)$ for non-degenerate $\overline t$, $\overline x$ and disjoint index sets $J$, $K$.

\label{s:cofibrant_replacment_math}

\subsection{Eilenberg--MacLane spaces}\label{s:emlspaces}

Given a group $\pi$ and an integer $n \geq 0$, an Eilenberg--MacLane space $\Kpin$ is a simplicial set satisfying
\[
\pi_k (\Kpin)= 
\left\{
	\begin{array}{ll}
		\pi & \mbox{for } k = n,\\
		0 & \mbox{else}.
	\end{array}
\right.
\]
In this text the symbol $\Kpin$ will always stand for the following concrete simplicial model, see~\cite[page 101]{may}
\[\Kpin_k = Z^n (\Delta^k; \pi),\] 
where $Z^n$ denotes the Abelian group of normalized cocycles. Similarly, we define the contractible space $WK(\pi, n)$ as
\[WK(\pi, n)_k = C^n (\Delta^k; \pi)\]
where $C^n$ denotes the Abelian group of normalized cochains. Since both are simplicial groups, they are fibrant.

According to \cite[Theorem 23.10]{may}, the universal principal bundle with fibre $K(\pi, n)$, i.e.
\[K(\pi, n) \hra WK(\pi, n) \xra{\delta} \overline WK(\pi, n),\]
has $\overline WK(\pi, n)$ isomorphic to $K(\pi, n+1)$ (a concrete isomorphism can be found in \cite{polypost}) and we will thus consider these spaces equal. The map $\delta$ is then the coboundary from the $n$-cochains to $(n+1)$-cocycles.

In the computational part, declaring the two spaces equal amounts to applying the canonical isomorphism and its inverse. These are given by straightforward formulas running in polynomial time, see \cite[Lemma~3.15]{polypost}.

\subsection{Principal twisted cartesian products}

Let $X$ be a simplicial set and $G$ a simplicial group. As in the preceding section, there is a universal principal bundle with fibre $G$

\[G \hra WG \xra{\delta} \overline W G.\]

A simplicial map $\tau \colon X \to \overline W G$ is known as a \emph{twisting function} and prescribes a \emph{principal twisted cartesian product} $X \times_\tau G \to X$, a simplicial analogue of a principal bundle. It is obtained by replacing one of the face operators in the usual cartesian product according to $\tau$, but we will not need to explain details here. There is an obvious extension to diagrams -- if $X$ is a diagram of simplicial sets and $G$ a diagram of simplicial groups, a twisting function $\tau \colon X \to \overline W G$ is then just a compatible family of twisting functions at each object and thus prescribes a ``compatible'' family of principal twisted cartesian products; explicitly, compatibility means that the canonical projection $X \times_\tau G \to X$ is a natural transformation, i.e.\ a map in $\Diag$. We also note that
\[\xymatrix{
X \times_\tau G \ar[r] \ar[d] & WG \ar[d]^-{\delta} \\
X \ar[r]^-\tau & \overline WG
}\]
is a map of principal bundles and is thus a pullback square.

\subsection{Principal bundles categorically}

For the purposes of a later generalization, we will define principal bundles categorically in any complete category $\mcC$ in the following way\footnote{%
	The definition will not capture surjectivity of the bundle projection, so that even the empty space over $X$ will be a principal $G$-bundle according to our definition.
}:

Let $G$ be a group object in $\mcC$. A \emph{$G$-torsor} is an object $P$ of $\mcC$ with a simply transitive (right) action of $G$. A regular action of $G$ on itself presents $G$ as a $G$-torsor. Thinking of the group additively, the simple transitivity is expressed as a difference map $P \times P \to G$, a generalization of the association $(A, B) \mapsto \overrightarrow{AB} = -A+B$ from the theory of affine spaces; we will use the nicer looking $B-A$ since we will have commutativity anyway. It is a simple matter to write down a set of axioms (e.g.\ $x + (y - x) = y$, as for affine spaces), each expressed as commutativity of a diagram involving finite products of $G$ and $P$. In particular, any functor that preserves finite products will automatically preserve group objects and their torsors.

\begin{example}
A non-empty torsor in the category of sets is a so-called heap, which we define later. Namely, for a heap $S$ and for any choice of zero $0 \in S$, the heap $S$ becomes a group, so that it possesses the regular right action on itself. For different choices, the induced groups are canonically isomorphic and the actions are identified under this isomorphism. The empty set is (in our definition) a torsor for any group.
\end{example}

A \emph{principal $G$-bundle} is a map $P \to X$, thought of as an object of the slice category $\mcC/X$, that is a torsor for the trivial group object $X \times G \to X$ in $\mcC/X$, given by the projection. A simple example of the preservation of torsors is the fact that principal $G$-bundles are closed under pullbacks -- the pullback functor $f^* \colon \mcC/X \to \mcC/Y$ clearly preserves all limits. Also, any functor from $\mcC$ preserving finite limits will preserve principal bundles, since the product in the slice category $\mcC/X$ is the pullback in $\mcC$. Explicitly, the structure maps for a principal bundle in terms of the category $\mcC$ are:
\begin{align*}
{+} & \colon G \times G \to G, & 0 & \colon * \to G, & {-} & \colon G \to G \\
{+} & \colon P \times G \to P, & & & {-} & \colon P \times_X P \to G
\end{align*}
where the action of $G$ on $P$ is required to be a map over $X$.

For a principal twisted cartesian product $P = X \times_\tau G \to X$, the last two maps are: the action
\[{+} \colon P \times G \to P,\quad (x, a) + g = (x, a + g)\]
(a map over $X$), and the difference
\[{-} \colon P \times_X P \to G, \quad (x, a) - (x, b) = a - b.\]

In particular, we will need principal bundles whose fibres are the diagrams of Eilenberg--MacLane spaces, which we describe next.

\subsection{Eilenberg--MacLane diagrams}

For a diagram $\pi \in \AbI$, we define a diagram of Eilenberg--MacLane spaces $\Kpin \in \Diag$ objectwise, i.e.\ by setting
\[\Kpin(i) = K(\pi(i), n).\]
Analogously, we define the diagram $WK(\pi, n)(i) = WK(\pi(i), n)$. Both these diagrams are fibrant (since they consist of fibrant objects).

The advantage of the concrete models described above is that maps to these diagrams can be identified with cochains and cocycles of the Bredon cochain complex. The following lemmas are easy generalizations of results in~\cite{may}.

\begin{proposition}\label{prop:EML-map}
Let $(X, A)$ be a pair of diagrams and let $\pi \in \AbI$. Then there are natural isomorphisms
\begin{align*}
\map((X, A), (W\Kpin, 0)) & \cong C^n(X, A; \pi), \\
\map((X, A), (\overline W\Kpin, 0)) & \cong Z^{n+1}(X, A; \pi).
\end{align*}
For a relative cell complex $(X, A)$ (or more generally for a cofibration $A \to X$), we also have
\[[X, \overline W\Kpin]^A \cong H^{n+1}(X, A; \pi),\]
where again maps on the left are fixed to be zero on $A$.
\end{proposition}

\subsection{Postnikov tower of a space} \label{s:Postnikov_tower_motivation}

We will give a very concise definition of a Postnikov tower of a space, mainly to explain that this definition has to be modified for diagrams; this case will then be treated in much more detail.

A Postnikov tower of a simply connected space $Y$ is a collection of maps $Y \to \Pnewst$ that display $\Pnewst$ as the result of killing homotopy groups of $Y$ above dimension $\then$. These approximations are organized in a tower
\[\cdots \to \Pnewst \to \Poldst \to \cdots \to \Pzerost;\]
i.e.\ the stages are connected by maps $\Pnewst \to \Poldst$. These are principal fibrations whose fibre is necessarily $\Lnst$, for $\pin$ the $\then$-th homotopy group of $Y$. In the standard model, they are even principal twisted cartesian products and as such are classified by a homotopy class $\kn \colon \Poldst \to \overline WK(\pin, \then)$, known as Postnikov invariant. One may then write
\[\Pnewst = \Poldst \times_{\kn} K(\pin, \then)\]
to get a very concrete inductive construction of the Postnikov tower, see \cite{polypost} for the algorithmic viewpoint. The Postnikov towers are employed in the algorithm by observing that $[X, Y]^A \cong [X, \Pnewst]^A$, for $\then \geq \dim X$, and also by relating $[X, \Pnewst]^A$ to $[X, \Poldst]^A$ via a long exact sequence that enables inductive computation.

The following problem occurs for diagrams: the diagram $\Poldst$ is not cofibrant in general and, as a result, the Postnikov invariant does not exist as an actual map $\Poldst \to \overline WK(\pin, \then)$, but rather as a map defined on its cofibrant replacement $\Poldst^\cofr$. As a result, the $\then$-th stage $\Pnewst$ is constructed as
\[\Pnewst = \Poldst^\cofr \times_{\kn} K(\pin, \then)\]
and will need to be cofibrantly replaced for the construction of $\towercomp{P}{\then+1}$ etc. Thus, for diagrams, a tower of the above simple shape must be replaced by a notion that incorporates cofibrant replacements. We will first define a general notion of such a tower and then give a precise definition of a Postnikov tower for a diagram of spaces.

\subsection{Towers}

\begin{sdefinition}
A \emph{tower} $T$ is a collection of diagrams $\Tm$, for $\them \geq 0$, together with maps $\Tm \to \Tmmoc$. A \emph{map of towers} $\varphi \colon S \to T$ is a collection of maps $\varphim \colon \Sm \to \Tm$ for which the square
\[\xymatrix@C=4pc{
\Sm \ar[r]^-{\varphim} \ar[d] & \Tm \ar[d] \\
\Smmoc \ar[r]_-{\cofrst{\varphimmo}} & \Tmmoc
}\]
commmutes for all $\them$. We denote by $\Tow$ the category of towers of diagrams.

An \emph{$\then$-restricted tower} is the collection of data as above, but with both $\Tm$ and $\Tm \to \Tmmoc$ defined only for $\them \leq \then$. The category of $\then$-restricted towers will be denoted $\Town$.
\end{sdefinition}

There is a pair of adjunctions
\[\xymatrix{
\towercomp{}{\then} \colon \Tow \ar@<.4em>[r]^-{\towercomp{}{\leq\then}} \ar@{}[r]|-{\bot} & \Town \ar@<.4em>[l]^-{\extn} \ar@<.4em>[r]^-{\towercomp{}{\then}} \ar@{}[r]|-{\bot} & \Diag \loc \susp{}{\then} \ar@<.4em>[l]
}\]
with the first top functor (left adjoint) restricting a tower to $\them \leq \then$ and the second associating to an $\then$-restricted tower the diagram sitting at the top level $\then$.

The bottom functors (right adjoints) are easily described as follows: The first one
\[{\extn \colon \Town \to \Tow}\]
extends the $\then$-restricted tower $T$ by iterated cofibrant replacements of $\towercomp{T}{\then}$ in such a way that the structure maps $\towercomp{T}{\them} \to \towercomp{T}{\them-1}^\cofr$ are the identity maps, for $\them > \then$. Since a $0$-restricted tower is exactly a diagram, we may view $\extz$ as a functor $\Diag \to \Tow$ and, from now on, we will not distinguish between a diagram $Z$ and its extension $\extz Z$ (consisting of iterated cofibrant replacements of the diagram $Z$). We will thus write
\[\Diag \subseteq \Tow.\]
In particular, the terminal diagram $\pt$ will be thought of as a tower in the following description. The second right adjoint $\Diag \to \Town$ sends a diagram $Z \in \Diag$ to the $\then$-restricted tower $T$ with $\towercomp{T}{\them} = \towercomp{\pt}{\them}$ for $\them < \then$ and $\towercomp{T}{\then} = \towercomp{\pt}{\then} \times Z$. We will not need a name for this functor, but the composite of the right adjoints will be denoted by $\susp{}{\then} \colon \Diag \to \Tow$ and is clearly right adjoint to the $\then$-th level functor $\towercomp{}{\then} \colon \Tow \to \Diag$ (the composite of the left adjoints). Also $\susp{}{0} = \extz$.

\begin{sdefinition}
We say that a tower $T$ is \emph{$\then$-truncated}, if it lies in the image of the extension functor $\extn$, i.e.\ if the structure maps $\Tm \to \Tmmoc$ are the identity maps for $\them > \then$. The \emph{$\then$-truncation} $\trunc{T}{\then}$ is the composite $\extn(\towercomp{T}{\leq \then})$ and admits a canonical map (the unit of the adjunction) $T \to \trunc{T}{\then}$.
\end{sdefinition}

\subsection{Homotopy groups of diagrams and towers}

Let $Y \in \Diag$ be a diagram of simply connected spaces. We denote by $\pi_n Y \in \AbI$ the diagram of the $n$-th homotopy groups of the spaces in the diagram $Y$. This makes sense since the $n$-th homotopy group is a functor $\pi_\then$ on simply connected spaces (it is however not a functor on all spaces). 

\begin{sdefinition}
For an $\then$-truncated tower $T$, we define its $j$-th homotopy group to be $\pi_j(T) = \pi_j(\towercomp{T}{\then})$. We note that $T$ is then also $\them$-truncated, for any $\them \geq \then$, and the definition of $\pi_j(T)$ is independent of $\them$.
\end{sdefinition}

\subsection{Principal bundles}

Let $P \to X$ be a principal $G$-bundle in the category of diagrams. For any $\then \geq 0$, the right adjoint $\susp{}{\then}$ preserves limits and, thus, $\susp{P}{\then} \to \susp{X}{\then}$ is a principal $\susp{G}{\then}$-bundle. Concretely, this involves actions of iterated cofibrant replacements of $\pt$ and $G$ and of course can be verified directly.

\subsection{Pullback of towers}

\begin{lemma} \label{lem:pullback_of_towers}
A square of $\then$-truncated towers
\[\xymatrix{
S \ar[r] \ar[d] & U \ar[d] \\
T \ar[r] & V
}\]
in which $\trunc{U}{\then-1} \xra\cong \trunc{V}{\then-1}$ is an isomorphism, is cartesian if and only if it is cartesian at each level $\them \leq \then$. Explicitly, this means $\towercomp{S}{\them} = \towercomp{T}{\them}$ for $\them < \then$ and $\towercomp{S}{\then} = \towercomp{T}{\then} \times_{\towercomp{V}{\then}} \towercomp{U}{\then}$.

In particular, $T \times \susp{Z}{\then}$ agrees with $T$ up to level $\then$ and its level $\then$ is $\towercomp{T}{\then} \times Z$.
\end{lemma}
\begin{proof}
The proof follows directly from the definition of pullback and the fact that pullbacks over diagrams are taken pointwise.
\end{proof}

\subsection{Postnikov tower for diagrams}\label{s:Postnikov_tower_diagrams}

Let $Y$ be a diagram of simply connected spaces. Letting $\pin \in \AbI$ be a diagram of Abelian groups (it will follow from the axioms that $\pi_n \cong \pi_n Y$, hence the name), we introduce the abbreviations
\[\Ln = \susp{\Lnst}{\then}, \quad
\En = \susp{\Enst}{\then}, \quad
\Kn = \susp{\Knst}{\then}.\]
As explained above, $\En \to \Kn$ is a principal $\Ln$-bundle.

\begin{definition}
A \emph{Postnikov system} of $Y$ is a map of towers $\varphi \colon Y \to P$, satisfying the following conditions for the $\then$-truncation $\Pnew = \trunc{P}{\then}$ and the associated $\varphi_\then \colon Y \to \Pnew$:

\begin{enumerate}\setcounter{enumi}{-1}
\item
For each $\then\ge 0$, there is given a diagram $\pin$ of Abelian groups.
\item
For each $\then\ge 0$,
\begin{enumerate}
\item
the induced map $\varphi_{\then*} \colon \pi_j(Y) \to \pi_j(\Pnew)$ is an isomorphism for $0\le j\le\then$,

\item
$\pi_j(\Pnew)=0$ for $j > \then$.

\end{enumerate}

\item\label{ei:pullback}
For each $\then\ge 0$, there is given a pullback square
\[\xymatrix{
\Pnew \ar[r]^-{\qn} \ar[d]_-{\pn} \pb & \En \ar[d]^-{\deltan} \\
\Pold \ar[r]_-{\kn} & \Kn
}\]
\end{enumerate}
\end{definition}

Towers $\Pnew$ are called \emph{stages} of the Postnikov system, and maps $\kn$ are called \emph{Postnikov classes} (the terms \emph{Postnikov factors} or \emph{Postnikov invariants} are also used in the literature). These are part of the structure of a Postnikov system.

We note that, since $\deltan \colon \En \to \Kn$ is a principal $\Ln$-bundle, so is its pullback $\pn \colon \Pnew \to \Pold$ and, in particular, there is an action $\Pnew \times \Ln \to \Pnew$ and a difference $\Pnew \times_{\Pold} \Pnew \to \Ln$.


For the sake of completeness, we also provide a description of the conditions in the definition in terms of the levels $\Pnewst$ of the Postnikov tower $P$. However, whenever possible the more compact and symmetric version with towers will be used.

\begin{lemma} \label{lem:Postnikov_tower_via_levels}
In terms of the levels $\Pnewst$ of the Postnikov tower $P$, the conditions are equivalent to:
\begin{enumerate}\setcounter{enumi}{-1}
\item
For each $\then\ge 0$, there is given a diagram $\pin$ of Abelian groups.
\item
For each $\then\ge 0$,
\begin{enumerate}
\item
the induced map $\varphi_{\then*} \colon \pi_j(\towercomp{Y}{\then}) \to \pi_j(\Pnewst)$ is an isomorphism for $0\le j\le\then$,

\item
$\pi_j(\Pnewst)=0$ for $j > \then$.

\end{enumerate}

\item
The $n$-the level $\Pnewst$ is a pullback in the following diagram
\[\xymatrix{
\Pnewst \ar[r]^-{\qn} \ar[d]_-{\pn} \pb & \Enst \ar[d]^-{\deltan} \\
\cofrst{\Poldst} \ar[r]_-{\kn} & \Knst
}\]
\end{enumerate}
\end{lemma}

\begin{proof}
The first point is clear and the second is an instance of Lemma~\ref{lem:pullback_of_towers}.
\end{proof}

We will only work with $\then$-truncated towers, where $\then = \dim X$. Clearly, a map between $\then$-truncated towers is the same as a map between the restricted towers $\towercomp{T}{\them}$, $\them \leq \then$. For this reason, it will be possible to represent towers and maps between them in a computer.

\begin{lemma}\label{l:maps_to_n_truncated}
%
A map $\varphi \colon S \to T$ from a 0-truncated tower $S$ to an $\then$-truncated tower $T$ is determined uniquely by the component $\towercompmap{\varphi}{\then}$. However, not every such map $\towercompmap{\varphi}{\then}$ determines a map of towers.
\end{lemma}

\begin{proof}
The components $\varphim$ with $\them > \then$ are determined from $\varphin$ by the $\then$-truncatedness of $T$, while the components $\varphim$ with $\them < \then$ by the 0-truncatedness
of $S$ and from the cofibrant replacement functor $\cofrst{(\ )}$ being faithful.
\end{proof}

\begin{theorem}\label{t:from_Y_to_Pn}
Let $(X, A)$ be a cellular pair. Then there is an isomorphism $[X, Y]^A \cong [X, \Pnewst]^A$ for $\then \geq \dim X$, where $\dim X$ is to be interpreted as the highest dimension of a cell in a cellular structure on $(X, A)$.
\end{theorem}

\begin{proof}
This is essentially the Whitehead theorem and the usual proof can be adopted. An abstract Whitehead theorem in model categories is proved in \cite[Theorem~2.2]{Vokrinek-algoheaps} and it applies here as well.
\end{proof}

We stress however that $\Pnewst$ is not fibrant and, thus, the homotopy classes are not represented by maps of diagrams. On the other hand, there is a model structure on the category of towers in which the Postnikov tower and its truncations $\Pnew$ are fibrant and thus, unlike for the levels $\Pnewst$, homotopy classes will be represented by actual maps of towers to $\Pnew$. We will not construct the model structure but give a direct proof of the representation theorem. To make this precise, for a diagram $X$ (i.e.\ a $0$-truncated tower), we specify the homotopy relation on maps of towers $X \to \Pnew$ to be the homotopy with respect to a cylinder object $\interval \times X$ (again a 0-truncated tower associated with $\interval \times X$ and where we remind our notation $I = \Delta^1$). The resulting set of relative homotopy classes will be denoted by $[X, \Pnew]^A$.

\begin{theorem}\label{t:representing_homotopy_classes}
Let $(X, A)$ be a cellular pair. Associating to a map of towers $\ell$ its $\then$-th component $\towercompmap{\ell}{\then}$ induces an isomorphism
\[[X, \Pnew]^{A} \xra\cong [\towercomp{X}{n}, \Pnewst]^{\towercomp{A}{n}} \xla\cong [X, \Pnewst]^A\]
on the sets of homotopy classes.

More precisely, given a homotopy class in $[X, \Pnewst]^A$ and a representative of its image in $[X, \Poldst]^A$ by a map of towers $X \to \Pold$ under $A$, there exists a lift $X \to \Pnew$, again a map of towers under $A$, that represents the original homotopy class.
\end{theorem}

\begin{proof}
As usual, it is sufficient to prove the existence part, since the uniqueness is simply the existence of a homotopy.

Firstly, we will construct special fibrant replacements of the Postnikov stages, $\Pnewst'$ of $\cofrst{\Pnewst}$ and $\fibrst{\Pnewst}$ of $\Pnewst$, fitting into the commutative diagram
\[\xymatrix{
\Pnewst \ar[r]^-\sim \ar[d] & \fibrst{\Pnewst} \ar[d] \\
\cofrst{\Poldst} \ar@{ >->}[r]^-\sim \ar[d] & \Poldst' \ar[d] \\
\Poldst \ar[r]_-\sim & \fibrst{\Poldst}
}\]
Proceeding inductively, we let $\Pzerost \xra{\sim} \fibrst{\Pzerost}$ be a fibrant replacement of $\Pzerost$, e.g.\ we can take the identity. In the inductive step, factor the composition $\cofrst{\Poldst} \xra\sim \Poldst \xra\sim \fibrst{\Poldst}$ into a trivial cofibration followed by a (necessarily trivial) fibration,
\[\cofrst{\Poldst} \cof[\sim] \Poldst' \fib \fibrst{\Poldst}.\]
This ensures that $\Poldst'$ is indeed fibrant, since it admits a fibration to a fibrant $\fibrst{\Poldst}$. Using that $\Knst$ is fibrant, we obtain a factorization of the Postnikov invariant $\kn$,
\[\kn \colon \cofrst{\Poldst} \cof[\sim] \Poldst' \to \Knst.\]
Now we take the pullbacks $\fibrst{\Pnewst}$ and $\Pnewst$ of the Eilenberg--MacLane fibration along the above factorization of the Postnikov invariant $\kn$:
\[\xymatrix{
	\Pnewst \ar[r]^-\sim \ar@{->>}[d] \pb & \fibrst{\Pnewst} \ar[r] \ar@{->>}[d] \pb & \Enst \ar@{->>}[d] \\
	\cofrst{\Poldst} \ar@{ >->}[r]^-\sim & \Poldst' \ar[r] & \Knst
}\]
This ensures that $\fibrst{\Pnewst}$ is indeed fibrant, since it admits a fibration to a fibrant $\Poldst'$.

%

We are now ready to prove the proposition. Let a homotopy class in $[X,\Pnewst]^A$ be represented by a map $\towercompmap{\psi}{\then} \colon \towercomp{X}{n} \to \fibrst{\Pnewst}$ under $\towercomp{A}{n}$. Let the image of this homotopy class in $[X,\Poldst]^A$ be represented by a map of towers $\ell \colon X \to \Pold$ under $A$ and consider the cofibrant replacement of its top component:
\[\towercomp{X}{n} = \cofrst{\towercomp{X}{n-1}} \xlra{\cofrst{\ellnmo}} \cofrst{\Poldst}.\]
Then the outer square in
\[\xymatrix{
	\towercomp{X}{n} \ar@/^1pc/[drr]^-{\towercompmap{\psi}{\then}} \ar@/_1pc/[ddr]_-{\cofrst{\ellnmo}} \ar@{-->}[dr] \\
	& \Pnewst \ar[r] \ar[d] \pb & \fibrst{\Pnewst} \ar@{->>}[d] \\
	& \cofrst{\Poldst} \ar[r] & \Poldst'
}\]
commutes up to homotopy under $\towercomp{A}{n}$, which enables us to replace the map $\towercompmap{\psi}{\then}$ by a map, homotopic under $\towercomp{A}{n}$, for which this square commutes strictly. Thus, it induces a map $\towercompmap{\ell}{\then} \colon \towercomp{X}{n} \to \Pnewst$ under $\towercomp{A}{n}$, and, together with the given $\ellm$, $\them < \then$, a map of towers $X \to \Pnew$, as desired.
\end{proof}

Putting these two theorems together, we observe that elements of $[X, Y]^A$ are represented by maps of towers $X \to \Pnew$ up to homotopy relative to $A$. It remains to relate $[X,\Pnew]^A$ with $[X, \Pold]^A$ to enable inductive computation.

For the principal $\Ln$-bundle $\Pnew \to \Pold$, we will derive an exact sequence of homotopy classes that involves also $[X, \Ln]^A$ and $[X, \Kn]^A$. Since we represent homotopy classes of maps to Postnikov stages by maps of towers, it will be convenient to do the same for maps into Eilenberg-MacLane diagrams:

\begin{lemma}\label{l:representing_homotopy_classes_EM}
Let $(X, A)$ be a cellular pair. There is an isomorphism
\[[X, \Ln]^A \cong [X, \Lnst]^A\]
and the homotopy classes are represented both by maps $X \to \Lnst$ of diagrams under $A$ and by maps $X \to \Ln$ of towers under $A$.
\end{lemma}

\begin{proof}
The representability on the level of diagrams follows from $\Lnst$ being fibrant. Then, by adjunction, we get the first isomorphism in
\[[X, \Ln]^A = [X, \susp{\Lnst}{\then}]^A \cong [\towercomp{X}{\then}, \Lnst]^{\towercomp{A}{\then}} \cong [X, \Lnst]^A,\]
the second follows from homotopy invariance.
\end{proof}

\subsection{Exact sequences}
\label{sec:exact_sequence}

We derive a general ``exact sequence'' that relates the sets of homotopy classes of maps to consecutive stages of a Postnikov tower and that does not depend on the choices of basepoints. As was explained in Section~\ref{s:idea_proof}, the action of $[X, \Lnst]^A$ on $[X, \Pnewst]^A$ has possibly non-trivial stabilizers and set of orbits possibly a proper subset of $[X, \Poldst]^A$ and the exact sequence captures both the stabilizers and the subset.

A \emph{sequence} is a diagram of the following shape
\begin{equation} \label{eq:abstract_exact_sequence}
\setone_\bullet \xlra{\connhom_\bullet} \settwo \acts \setthree \xlra{s} \setfour \xlra{t} \setfive
\end{equation}
where $\setthree$, $\setfour$ are sets, $\setfive$ a pointed set with basepoint $0 \in \setfive$, $\settwo$ a group and $\setone$ a collection of groups $\setone_\eltfour$ indexed by $\eltfour \in \setfour$. The maps $s$ and $t$ are maps of sets, the arrow at $\setthree$ denotes an action of $\settwo$ on $\setthree$  and $\connhom_\bullet$ is a collection of group homomorphisms $\connhom_\eltthree \colon \setone_{s(\eltthree)} \to \settwo$ indexed by $\eltthree \in \setthree$.

\begin{remark}
In fact, the groups $\setone_\bullet$ and group homomorphsisms $\connhom_\bullet$ will not be indexed by elements of $\setthree$ and $\setfour$, but rather by elements of some bigger sets $\calsetthree$ and $\calsetfour$ that surject onto $\setthree$ and $\setfour$. Mathematically, this does not change anything, since the image of $\connhom_\bullet$ does not depend on the representative in $\calsetthree$ and this will be the main object, by the following definition.
\end{remark}

\begin{definition} \label{d:exactness}
We say that the above sequence is \emph{exact} if
\begin{itemize}[labelindent=.5em,leftmargin=*,label=$\bullet$,itemsep=0pt,parsep=0pt,topsep=0pt]
\item
	$t^{-1}(0) = \im s$,
\item
	$s(\eltthree) = s(\eltthree')$ if and only if $\eltthree$, $\eltthree'$ lie in the same orbit of the $\settwo$-action, i.e.\ $\eltthree + \elttwo = \eltthree'$ for some $\elttwo \in \settwo$, and
\item
	the stabilizer of $d \in \setthree$ is exactly the image of $\connhom_\eltthree$.
\end{itemize}
\end{definition}

We may construct out of this sequence an ordinary exact sequence of pointed sets in the following way: choose a basepoint $\eltthree \in \setthree$ and then consider
\[\setone_{s(\eltthree)} \xlra{\connhom_\eltthree} \settwo \xlra{\actmap} \setthree \xlra{s} \setfour \xlra{t} \setfive\]
with $\actmap(\elttwo) = \eltthree + \elttwo$, the action of $\settwo$ on the fixed element $\eltthree$. It is easily seen to be really exact, where $\setone_{s(\eltthree)}$ and $\settwo$ are equipped with the respective zeroes as basepoints, $\setthree$ with basepoint $\eltthree$, $\setfour$ with basepoint $s(\eltthree)$ and $\setfive$ with the given element $0 \in \setfive$.

%
%

\subsection*{Exact sequence relating consecutive stages}

By composing $\alpha \colon A \to Y$ with various maps in the Postnikov tower of $Y$, we make all
\[\Pnew, \Pold, \En, \Kn\]
into towers under $A$, i.e.\ objects of $A/\Tow$. Further, $\Ln$ is considered as a tower under $A$ via the constant map onto the zero of $\Ln$; more precisely, the constant map onto the zero $\towercomp{A}{\then} \to \Lnst$ is adjoint to the required $A \to \Ln$. For the purpose of the description of the exact sequence, we will denote maps $X \to \Pold$ by $\ellnmo$, $\ellnmo'$, etc.\ and maps $X \to \Pnew$ by $\elln$, $\elln'$, etc.

Our main exact sequence is
\begin{equation}\label{eq:general_exact_sequence}
[I \times X,\Pold]^{\partial}_\bullet \xlra{\connhom_\bullet} [X,\Ln]^A \acts [X,\Pnew]^A \xlra{\pnst} [X,\Pold]^A \xlra{\knst} [X,\Kn]^A,
\end{equation}
whose only non-trivial object is the collection of groups
\[[I \times X,\Pold]^{\partial}_\bullet := [I \times X,\Pold]^{(\partial I \times X) \cup (I \times A)}_\bullet,\]
indexed by $\eltfour \in [X, \Pold]^A$, where for each such $\eltfour = [\ellnmo]$, the corresponding group $[I \times X,\Pold]^{\partial}_e$ is the group of homotopy classes fixed on each copy of $X$ by $\ellnmo$ and on $I \times A$ by the constant homotopy at the given map $A \to \Pold$. (As explained in the above remark, this collection is indexed by actual maps $\ellnmo \colon X \to \Pold$ rather than the homotopy classes $[\ellnmo]$; this will be important later in the computational part.) The element $0 \in [X, \Kn]^A$ (the basepoint) is the only homotopy class in the image of $\deltanst \colon [X, \En]^A \to [X, \Kn]^A$ (since $\En$ is contractible, there is a unique homotopy class $X \to \En$, see Lemma~\ref{lem:En_contractible} for a more precise statement and proof).

The maps $\pnst$ and $\knst$ are induced by $\pn$ and $\kn$, respectively. The action is also induced by the action of $\Ln$ on $\Pnew$. It remains to describe the homomorphisms
\[\connhom_{[\elln]} \colon [I \times X,\Pold]^{\partial}_{[\ellnmo]} \to [X, \Ln]^A,\]
where $\ellnmo =\pnst(\elln)$. Starting with a homotopy $h \colon I \times X\ra\Pold$ as above, lift it to a homotopy $\widetilde h \colon I\times X\ra\Pnew$ starting at $\elln$ and relative to $A$. The restriction $\widetilde h|_{\vertex1\times X}$ is then of the form $\elln + \zeta$ for a unique $\zeta \colon X \to \Ln$ (namely, $\zeta$ is the difference $\widetilde h|_{1 \times X} - \elln$) and we set $\connhom_{[\elln]}[h] = [\zeta]$.

\begin{proposition} \label{prop:connecting_homomorphism}
The above is a well defined exact sequence.
\end{proposition}

\begin{proof}
The proof in \cite[Section~5]{FilVokri} applies to any principal bundle with a homotopy lifting property for the pair $(X, A)$, such as $\Ln \to \Pnew \to \Pold$.
\end{proof}

\subsection{Heaps} \label{sec:heaps}

In the stable situation $\dim X \leq 2 \operatorname{conn} Y$, the set $[X, Y]^A$ is actually an Abelian heap (this is proved later in Theorem~\ref{thm:stable_situation_heaps}) and we will exploit this structure for the computations. We start with a formal definition of a heap. Intuitively, a heap is a group without a definite choice of zero, so that one has addition with respect to an arbitrary zero.

\begin{sdefinition}
A \emph{heap} is a set $S$ with a ternary operation, denoted by $x +_p y$ in this paper, that satisfies the identity law
\[x +_p p = x = p +_p x\]
together with a ``partial para-associative law,'' or just associative law,
\[(x +_p y) +_q z = x +_p (y +_q z).\]
It is said to be \emph{Abelian} if
\[x +_p y = y +_p x.\]
\end{sdefinition}

Given $p \in S$, we obtain a group structure on $S$ with zero $p$, addition given by $x + y = x +_p y$ and inverse $-x = p +_x p$; we denote this group by $S_p$. It is Abelian if and only if the heap $S$ is Abelian. A different choice of the zero element leads to an isomorphic group, the isomorphism being the translation map $S_p \to S_q$, $x \mapsto x +_p q$. We will not work with heaps directly, but rather we will choose a zero and work with the induced group.

\subsection{Exact sequences of heaps}

Thus, if an exact sequence in the sense of Definition~\ref{d:exactness} consists of Abelian heaps and heap homomorphisms, by choosing basepoints, we obtain an ordinary exact sequence of Abelian groups. Since computations with exact sequences of Abelian groups (and known homomorphisms) are possible, this finishes our mathematical description of the computation of $[X,\Pnew]^A$, once we explain how this is an Abelian heap.

\subsection{Stability and Abelian heaps}


The (unreduced) suspension $\Sigma Y$ of a diagram $Y$ is the quotient of $\Delta^1 \times Y$ under the identification that squashes each of $0 \times Y$ and $1 \times Y$ separately to a point, i.e.\ it is the diagram of unreduced suspensions.

\begin{theorem} \label{thm:stable_situation_heaps}
Let $(X, A)$ be a cellular pair. When $\dim X \leq 2 \operatorname{conn} Y$, there is a bijection
\[[X, Y]^A \cong [\Sigma X, \Sigma Y]^{\Sigma A}\]
and the set on the right admits a canonical structure of an Abelian heap.
\end{theorem}

\begin{proof}
This is essentially contained in the proof of \cite[Theorem~1.1]{Vokrinek-algoheaps} applied to the category $\mcM = A / \Diag$ of diagrams under $A$. More precisely, it is proved in that theorem that $[X, Y]^A$ is isomorphic to $[\Sigma X, \Sigma Y]^{\partial I\!\!I}$, where $\partial I\!\!I$ is the suspension of the initial object, i.e.\ of $A$, and is thus $\Sigma A$. Both $\Sigma X$ and $\Sigma Y$ are made into diagrams under $\Sigma A$ in an obvious way by suspending the given maps $A \to X$ and $A \to Y$.

The second statement is a part of \cite[Theorem~1.1]{Vokrinek-algoheaps} and will be explained in greater detail in the proof of Theorem~\ref{thm:weakly_locally_effective_heap_structure}.
\end{proof}

\section{Algorithmic structures on mathematical objects}\label{s:Algorithmic_structures}

In this section, we deal with algorithmic aspects of the mathematical objects treated in the previous section. First we present our point of view on computations in/with an object like a simplicial set $X$ and introduce various levels of its computability -- (weakly) locally effective, effective (and later also homologically effective). As the running time analysis of computations of invariants like $[X,Y]^A$ for a \emph{single} instance makes little sense, we will have to deal, at least implicitly, with \emph{families} of inputs for this purpose and this introduces a further layer of complexity into the picture. For this reason, we postpone this undertaking to the very end of this section.

\subsection{Computations in objects vs computations with objects}

We would like to point out a qualitative difference between two computational problems concerning a simplicial set. The first task is to compute the $j$-th face or degeneracy of a given simplex and the second is to compute the $n$-th homology group.

The first problem can quite often be tackled without exact knowledge of the simplicial set in question, e.g.\ it is computed in exactly the same way in a space and in any of its subspaces; thus, it only concerns a ``neighbourhood'' of the given simplex and that is why we call it local. If all such local computations are available (in this case, faces and degeneracies), we call a simplicial set \emph{locally effective}; we give a precise definition later. It is simple to give a similar definition for any algebraic structure -- all operations should be computable, e.g.\ addition, zero and inverse in a locally effective Abelian group etc. In general, we speak of locally effective objects.

On the other hand, the second problem of computing $H_n$ concerns the whole simplicial set and is thus global. Provided that $X$ is locally effective and that we are given a list of all its (non-degenerate) simplices, we call $X$ \emph{effective} and for such $X$ it is possible to compute basically anything, including the homology groups. This should be viewed as the strongest version of (global) effectiveness of $X$.

A general definition of an effective object has the following scheme: We declare certain locally effective objects to be ``standard effective objects'', e.g.\ in the case of Abelian groups these are the products of cyclic groups, and in the model categorical cases these are the cell complexes whose elements are represented uniquely using cells as in Propositions~\ref{prop:cellular_diagrams_spaces} and~\ref{prop:cellular_diagrams_complexes}. Then a general locally effective object is effective if there is provided a computable isomorphism with a standard effective object.

In the last part of the paper, we will use heavily \emph{homologically effective chain complexes} -- these are generalizations of effective chain complexes, where an isomorphism with a standard effective object is replaced by a chain homotopy equivalence.

\subsection{Weak local effectiveness}

There is one more issue that was not apparent for simplicial sets. Our main object is $[X, Y]^A$, the set of homotopy \emph{classes} of maps and our algorithms naturally work with actual maps, i.e.\ \emph{non-unique} representatives of these homotopy classes. We also mention a much simpler example of the group $\bbZ/n$ where, for example, on the level of representatives the addition is the usual addition of integers, regardless of $n$, and is thus no different from~$\bbZ$. Of course, one can make representatives unique if one chooses a set of preferred representatives and output only these in all algorithms; this may sound natural for $\bbZ/n$ but there are no obvious preferred representatives in $[X, Y]^A$ so that the non-unique representation is unavoidable. When speaking about local computations, we were silently assuming that representatives were unique, in which case there is an easy way of distinguishing $\bbZ/n$ from~$\bbZ$: pick any non-zero element (assume it is given), multiply it by $n$ and check whether the result is zero. The same works with non-unique representatives if equality is decidable, i.e.\ if there is provided an algorithm testing whether two inputs represent the same element. If this is the case and all local computations are available, we still call the object locally effective.

Thus, in order to make $[X, Y]^A$ locally effective (even as a set), we would require an algorithm testing whether two maps are homotopic. In fact, such an algorithm exists even non-stably (by the methods of~\cite{FilVokri}) but is not needed in this paper. We will thus also work with structures where local computations are possible but equality is not necessarily decidable. We call such structures \emph{weakly locally effective}. (Previously, these were called semi-effective, but since they are weaker than locally effective ones, we decided to change the name.)

We remark that there will be no weakly effective objects, since equality will be decidable in all our standard effective objects and, for effective objects, one can transfer the equality problem along the given isomorphism to a standard effective object.

\subsection{Preview on running times}

Given that an effective Abelian group is a collection of algorithms we would like to clarify on claims concerning the running time of computing $[X, Y]^A$ or, more precisely, the running time of the algorithm giving the isomorphism type of this Abelian group. Practically, and in accordance with our partial implementation within the framework of object-oriented programming, such a computation usually splits into the ``construction'' of the object $[X, Y]^A$ itself (in OOP terms the call of the constructor) and the call of the responsible function (in OOP terms the call of the ``method''); what matters, of course, is the total running time.

We do not specify how to split the computation. One of the options is the lazy implementation where nothing is computed before it is needed. In this extreme case, the construction running time is zero and the running time of the method is the only relevant part. However, in this approach, any required data involving any intermediate step is computed repeatedly from a scratch, and so this does not prescribe a very practical algorithm. In the opposite extreme, the isomorphism type etc.\ can be computed upon the construction and outputting it via the method then takes very little time. To make our running time analysis simpler, and only for this reason, we will be assuming the lazy implementation, so that the construction time does not enter the analysis.\footnote{%
A more detailed analysis allowing non-lazy implementations was developed in \cite{polypost} and further in~\cite{aslep}. Following this formalism, one can show that even the non-lazy implementation runs in polynomial time, although we believe that the reader should consider it rather clear that its running time is not greater than that of the lazy implementation.
}
Thus, our way to prove the polynomiality claim will be to show recursively that all the algorithms comprising any involved computational object (e.g.\ the Postnikov stage $\Pnew$) run in polynomial time \emph{provided} that the same is true for all algorithms of all objects used inside this object (e.g.\ the previous Postnikov stage $\Pold$). We will elaborate on this at the end of this section, but it might be helpful to have this goal in mind already now.

\subsection{Weakly locally effective sets}

Let $A$ be a set. We say $A$ is \emph{weakly locally effective} if there is given a set $\mcA$ and a surjective map $\mcA \to A$ (a \emph{weakly locally effective} representation), denoted $\alpha \mapsto [\alpha]$, in such a way that elements of $\mcA$ have a specified representation in a computer (for definiteness, we might assume that elements of $\mcA$ are actual bit strings, but we will not go into such details). A mapping $f \colon A \to B$ between weakly locally effective sets is said to be \emph{computable} if there is given an algorithm that computes a mapping $\varphi \colon \mcA \to \mcB$ that represents $f$, i.e.\ such that $f([\alpha]) = [\varphi(\alpha)]$.

\subsection{Locally effective sets}

We say that the representation of $A$ is \emph{locally effective} if there is provided an algorithm that, given $\alpha, \beta \in \mcA$, decides whether $[\alpha] = [\beta]$.

One of the possibilities, occuring frequently in this paper, is that the representation map $\mcA \to A$ is bijective, i.e.\ that any element of $A$ has a unique representative.

\subsection{(Weakly) locally effective surjections}

Before explaining the algebraic examples, we mention a general principle in the computational world: existence should be replaced by computability. This will not be of much concern to us, since algebraic structures are defined by equalities, but when dealing with exactness, surjectivity is crucial. In ordinary mathematics, a mapping $f \colon A \to B$ is surjective if
\[\forall b \in B \  \exists a \in A \colon f(a) = b.\]
In the computational world, we thus require an algorithm computing, for any $b \in B$, some preimage $a \in f^{-1}(b)$. In addition, in the weakly locally effective setup, this is handled on the level of representatives, so that the algorithm computes, for any representative $\beta \in \mcB$, a representative $\alpha \in \mcA$ of its preimage, i.e.\ $f([\alpha]) = [\beta]$. We remark, that this computable mapping does not, in general, prescribe a mapping $B \to A$, i.e.\ it may happen that $[\beta] = [\beta']$, while for the computed preimages $[\alpha] \neq [\alpha']$. For this reason, the computable mapping $\mcB \to \mcA$ will be called a \emph{weak section} of $f \colon A \to B$.

To summarize, we may say that $f$ is a surjection in the effective setting, if it admits a computable weak section.

\subsection{(Weakly) locally effective algebraic structures}

An algebra is a collection of sets and operations among them satisfying certain identities (i.e.\ an object of some variety of multi-sorted algebras). We then say that it is (weakly) locally effective if all the involved sets are (weakly) locally effective and if all operations are computable. We will now give a detailed definition for the structures used in this paper.

\subsection{Abelian groups}\label{sec:effabgrp}

A \emph{weakly locally effective} Abelian group $A$ is a weakly locally effective set for which the zero, addition and inverse are computable. In more detail we can compute $o\in \mathcal{A}$ such that $[o] = 0$, given any $\alpha, \beta \in \mathcal{A}$ we can compute $\gamma \in \mathcal{A}$ such that $[\gamma] = [\alpha] + [\beta]$ and given any $\alpha \in \mcA$ we can compute $\beta \in \mcA$ such that $[\beta] = - [\alpha]$,

A weakly locally effective Abelian group is $A$ \emph{effective} if there is given an isomorphism $A\cong\bbZ/q_1\oplus\cdots\oplus\bbZ/q_r$, computable together with its inverse. In detail, this consists of
\begin{itemize}
\item
an algorithm that outputs a finite list of generators $a_1,\ldots,a_r$ of $A$ (given by representatives) and their orders $q_1,\ldots,q_r\in\{2,3,\ldots\}\cup\{0\}$ (where $q_i=0$ gives $\bbZ/q_i=\bbZ$),
\item
an algorithm that, given $\alpha\in\mcA$, computes integers $z_1,\ldots,z_r$ so that $[\alpha]=\sum_{i=1}^rz_i a_i$; each coefficient $z_i$ is unique within $\bbZ/q_i$.
\end{itemize}

We will utilize the following lemmas; they were originally given in \cite{post}.

\begin{lemma}[kernel and cokernel]\label{l:ker_coker}
Let $f\colon A\to B$ be a computable homomorphism
of effective Abelian groups.
Then both $\ker(f)$ and $\coker(f)$ can be represented as effective Abelian groups.
\end{lemma}

For the second lemma, we need a definition of exactness in the computational setting. Assuming $g \circ f = 0$, the exactness of a sequence
\[\xymatrix{A \ar[r]^{f} & B \ar[r]^{g} & C},\]
means surjectivity of the restricted map $f \colon A \to \ker g$. By the above, this condition should be replaced by a computable weak section and we arrive at the following definition. \label{s:effective_exactness_groups}

\begin{definition}
A \emph{weakly locally effective short exact sequence} is an exact sequence
\[\xymatrix{0 \ar[r]& A \ar[r]^{f} & B \ar[r]^{g} & C \ar[r]&0}\]
consisting of weakly locally effective Abelian groups and computable homomorphisms together with computable mappings $\sigma$ and $\rho$ such that:
\begin{itemize}
\item
$\sigma \colon \mcC \to \mcB$ such that $g([\sigma(\gamma)])=[\gamma]$ for all $\gamma \in \mcC$,
\item
$\rho \colon \mcB \to \mcA$, defined only on representatives of $\ker g$, such that $f([\rho(\beta)])=[\beta]$.
\end{itemize}
\end{definition}

\begin{lemma}[short exact sequence]\label{l:ses}
There is an algorithm that, given a weakly locally effective short exact sequence
\[\xymatrix{0 \ar[r]& A \ar[r]^{f} & B \ar[r]^{g} & C \ar[r]&0}\]
with $A$ and $C$ effective, supplies an effective representation of~$B$.
\end{lemma}

\begin{lemma}[preimage]\label{l:preimage}
Let $f \colon A \to B$ be a computable homomorphism of effective Abelian groups. Then there is an algorithm that, given $b \in B$, decides whether it lies in $\im f$. If it does, it computes a preimage $a \in f^{-1}(b)$.
\end{lemma}
\begin{proof}
Compute the images $f(a_1), \ldots, f(a_r)$ of the generators of $A$. Next, decide if the equation
\[x_1 f(a_1) + \cdots + x_r f(a_r) = b\]
has a solution (this is done by translating to the direct sum of cyclic groups and solving there using the standard methods). If a solution exists, output $a = x_1 a_1 + \cdots + x_r a_r$.
\end{proof}

Later, the following lemmas will be used in the computation of Bredon cohomology which, in turn, will be useful in describing maps into Postnikov stages.

\begin{lem}\label{lem:minihom}
Let $A,B$ be effective Abelian groups. Then $\Hom (A,B)$ is an effective Abelian group.
\end{lem}

\begin{proof}
The proof is not complicated. We will only need the case of $A$ being free Abelian so that $\Hom(A, B)$ is a product of copies of $B$ and the result is trivial.
\end{proof}


Let $\mcI$ be a fixed finite category and let $\pi \in \AbI$ be a diagram such that every $\pi(i)$ is effective Abelian and every morphism is a computable homomorphism. We then say that $\pi$ is an \emph{effective diagram of Abelian groups}. As a consequence of the previous lemma, we get:

\begin{lemma}\label{lem:hom}
Let $\mcI$ be a fixed finite category and let $\pi, \rho \in \AbI$ be effective diagrams of Abelian groups. Then $\Hom_{\AbI}(\rho, \pi)$ is an effective Abelian group.
\end{lemma}

\begin{proof}
Clearly $\Hom(\rho, \pi)$ is the kernel of the homomorphism
\[F \colon  \prod_{i \in \mcI} \Hom(\rho(i), \pi(i)) \to \prod_{f\colon i_0 \to i_1}\Hom(\rho(i_1), \pi(i_0))\]
given by
\[F(g_i)_{i \in \mcI} = \big(\pi(f)g_{i_1} - g_{i_0}\rho(f)\big)_{f \colon i_0 \to i_1}\]
and as such is effective according to Lemmas~\ref{lem:minihom}, \ref{l:ses} and~\ref{l:ker_coker} (dealing with $\Hom$, finite products and kernel, respectively).
\end{proof}

\subsection{Simplicial sets} \label{s:math_simplicial_sets}

For simplicial sets, we will assume that the representive of each simplex is unique (though, decidable equality should be sufficient).

\begin{Def}
Let $X$ be a simplicial set. We say that $X$ is \emph{locally effective} if the underlying sets $X_n$ of $n$-simplices are locally effective and algorithms are provided computing the faces and degeneracies of any given simplex of $X$.
\end{Def}

We will now use the canonical cellular structure of a simplicial set $X$ to describe a standard representation of a finite simplicial set. We recall that cells are exactly the non-degenerate simplices of $X$ and that any simplex $x$ can be uniquely written as $x = s_J \overline x$, a degeneracy of a non-degenerate simplex $\overline x$. We may then represent $x$ as a pair $(J, \overline x)$, where it is simple to come up with an encoding of a finite index set like $J$ and of a finite number of non-degenerate simplices like $\overline x$. In order to describe the cellular structure, we need to specify the attaching maps, i.e.\ for each non-degenerate simplex $\overline x$, we need to prescribe each of its faces: $d_i \overline x = s_J \overline y$. This then gives enough information for the computation of faces and degeneracies of arbitrary simplices, usign the simplicial identities.

We might call a simplicial set given as above, i.e.\ via a list $e_\alpha$ of its non-degenerate simplices and lists of identities of the form $d_i e_\alpha = s_J e_\beta$ with $e_\alpha$ and $e_\beta$ non-degenerate, a standard effective simplicial set. According to this choice, an \emph{effective simplicial set} is a locally effective simplicial set equipped with an isomorphism, computable in both directions, with a simplicial set as above. Explicitly, this means: there exists an algorithm producing a list of all non-degenerate simplices\footnote{%
	It is also possible to ask only for the list of all non-degenerate simplices of any given dimension (passed as the input), leading to what mathematicians would call a locally finite simplicial set (for a meaning of local different from ours).
} and an algorithm that expresses any given simplex as a degeneracy of a non-degenerate one (though, such an algorithm already follows from local effectivity).

The notion of local effectiveness specializes to diagrams of simplicial sets in the following way:

\begin{Def}
Let $\mcI$ be a finite category. We say that a diagram $X \in \Diag$ is \emph{locally effective}, if, for any object $i$ of $\mcI$, the simplicial set $X(i)$ is locally effective and, for any morphism $f$ of $\mcI$, the map $X(f)$ is computable.
\end{Def}


\begin{definition}
A cellular pair $(X,A)$ of diagrams of simplicial sets is \emph{effective} if $X$ is locally effective (and hence also $A$) and there is given
\begin{itemize}
\item
an algorithm that outputs a finite list of cells $e_\alpha \in X(i_\alpha)$, $\alpha \in \mcA$,
\item
an algorithm that, given a simplex $e \in X \smallsetminus A$, computes the unique expression
\[e = s_J(f^*(e_\alpha))\]
of Proposition~\ref{prop:cellular_diagrams_spaces}.
\end{itemize}
\end{definition}

There is a completely analogous definition of a \emph{pointwise effective} diagram, where the cells generate the individual simplicial sets in the diagram separately, i.e.\ the expression in the second point is replaced by $e = s_J(e_\alpha)$.

\subsection{Chain complexes}

A chain complex $C$ \emph{locally effective} if all the chain groups $C_n$ are locally effective Abelian groups and the differentials are computable. Analogously, a diagram $C \in \ChI$ of chain complexes is \emph{locally effective} if $C(i)$ is locally effective for every object $i$ of $\mcI$ and if $C(f)$ is a computable homomorphism for every morphism $f$ of $\mcI$.

\begin{definition}
A cellular pair $(C,C')$ of diagrams of chain complexes is \emph{effective} if $C$ is locally effective (and hence also $C'$) and there is given
\begin{itemize}
\item
an algorithm that outputs a finite list of cells $c_\alpha \in C(i_\alpha)$, $\alpha \in \mcA$,
\item
an algorithm that, given a chain $c \in C$, computes the unique expression
\[c = \sum_{\alpha \in \mcA; f \colon i \to i_\alpha} k_{\alpha, f} f^*(c_\alpha) \mod C'\]
of Proposition~\ref{prop:cellular_diagrams_complexes}.
\end{itemize}
\end{definition}



The following lemma follows easily from the definitions given in this section.
\begin{lemma}
Let $(X,A)$ be a locally effective pair of diagrams of simplicial sets $A,X \in  \Diag$ and let $\rho \in \AbI$ be an effective diagram of Abelian groups. The following holds:
\begin{enumerate}
\item If $(X,A)$ is effective, then $C_*(X,A)$ is effective.
\item If $(X,A)$ is effective, then $C^n _\mcI (X, A; \rho)$, $Z^n _\mcI (X, A; \rho)$ and $H^n _\mcI (X, A; \rho)$ are effective Abelian groups.
\end{enumerate}
\end{lemma}

We remark that $C^*_\mcI (X,A;\rho)$ is not an effective chain complex according to our definition since it does not consist of free Abelian groups.

\subsection{Eilenberg-MacLane diagrams} \label{s:EM_diagrams_comp}

Since the Eilenberg-MacLane diagram $\Kpin$ has as $\then$-simplices $(\Kpin)_\then \cong \pi$, for a locally effective $\Kpin$ the coefficient system $\pi$ must be locally effective, too. In addition $\pi \cong H_\then(\Kpin)$, so that the (not yet defined) pointwise homologically effective $\Kpin$ will have $\pi$ effective; the converse is also true (this is \cite[Theorem~3.16]{polypost}), but not needed. Since we do not want to introduce another name for locally effective diagrams $\Kpin$ with $\pi$ effective, we will call them pointwise homologically effective; until the construction of the Postnikov tower, the reader may consider these synonymous. The isomorphism
\[\map((X, A), (\overline W \Kpin, 0) \cong Z^{\then + 1}(X, A; \pi)\]
is computable in both directions and so is the one for cochains. When $\pi$ is effective, we may decide if a cocycle is a coboundary (cochain groups are effective and the differential is computable), so that we may also decide whether a given map $X \to \overline W\Kpin$, zero on $A$, lifts to a map $X \to W\Kpin$, zero on $A$. This will be a crucial ingredient for the computational version of the obstruction theory.

\subsection{Cofibrant replacements}

As mentioned in Section~\ref{s:cofibrant_replacment_math}, the cofibrant replacement of any diagram is cellular. By the explicit description of the cells, it is clear that the cofibrant replacement of a pointwise effective diagram is effective. We will later see a variation of this result in Proposition~\ref{prop:cofibrep} -- the cofibrant replacement of a pointwise homologically effective diagram is homologically effective.

\subsection{Towers}

We will be working only with $\then$-truncated towers, for $\then$ fixed. In this situation, we may replace all $\then$-truncated towers by $\then$-restricted towers. The locally effective towers then have the obvious definition. Of course, locally effective (non-truncated) towers can be defined easily too.

\subsection{Postnikov towers}

Theorem~\ref{t:postnikov-restat} constructs a pointwise homologically effective $\then$-restricted Postnikov tower. With the exception of the proof of this theorem, we will only use local effectiveness of the tower and the effectiveness of homotopy groups, as explained in Section~\ref{s:EM_diagrams_comp} -- the full strength of pointwise homological effectiveness is employed in the inductive construction of the tower.

\subsection{Weakly locally effective exact sequences} \label{sec:abstract_exact_sequence}

A \emph{weakly locally effective collection of groups} $\setone_\bullet$ is a collection of groups $\setone_\caleltfour$, indexed by $\caleltfour \in \calsetfour$, together with surjections $\calsetone_\caleltfour \to \setone_\caleltfour$ that, together, provide a weakly locally effective representation
\[\calsetone = \coprod \calsetone_\caleltfour \to \coprod \setone_\caleltfour\]
Addition in these groups is represented by a computable map $\coprod \calsetone_\caleltfour \times \calsetone_\caleltfour \to \coprod \calsetone_\caleltfour$, i.e.\ $\calsetone \times_\calsetfour \calsetone \to \calsetone$, etc. In other words, a single algorithm is required, computing the addition in all the groups in the collection.

Similarly, an \emph{effective collection of groups} is a weakly locally effective collection of groups that possesses, in addition, an algorithm that computes for any given $\caleltfour \in \calsetfour$ a set of generators of $\setone_\caleltfour$ together with their orders and also an algorithm that computes the expression of any element of $\calsetone$ as an integral combination of these generators.

A \emph{computable collection of group homomorphism}s $\connhom_\caleltthree \colon \setone_{\sigma(\caleltthree)} \to \settwo$, indexed by $\caleltthree \in \calsetthree$, is similarly represented by a computable map $\widetilde\connhom \colon \calsetthree \times_\calsetfour \calsetone \to \calsettwo$, taking a pair $(\caleltthree, \caleltone)$ to a representative of $\connhom_\caleltthree([\caleltone])$.

A \emph{weakly locally effective sequence} is
\[\setone_\bullet \xlra{\connhom_\bullet} \settwo \acts \setthree \xlra{s} \setfour \xlra{t} \setfive\]
where $\setthree$, $\setfour$ are weakly locally effective sets, $\setfive$ a weakly locally effective pointed set with basepoint $[o] \in \setfive$, $\settwo$ a weakly locally effective group and $\setone_\bullet$ a weakly locally effective collection of groups $\setone_\caleltfour$ indexed by $\caleltfour \in \calsetfour$. The maps $s$ and $t$ are computable maps of sets, represented by $\sigma$ and $\tau$, the arrow at $\setthree$ denotes a computable action of $\settwo$ on $\setthree$  and $\connhom_\bullet$ is a computable collection of group homomorphisms $\connhom_\caleltthree \colon \setone_{\sigma(\caleltthree)} \to \settwo$ indexed by $\caleltthree \in \calsetthree$.

A \emph{weakly locally effective exact sequence} is a weakly locally effective sequence in which the following algorithms are provided, parallel to Definition~\ref{d:exactness}:
\begin{itemize}
\item
	for $\varepsilon \in \mcE$ such that $t[\varepsilon] = 0$, compute $\delta \in \mcD$ such that $s[\delta] = [\varepsilon]$,
\item
	for $\delta,\, \delta' \in \mcD$ such that $s[\delta] = s[\delta']$, compute $\beta \in \mcH$ such that $[\delta] + [\beta] = [\delta']$,
\item
	for $\delta \in \mcD$ and $\beta \in \mcH$ such that $[\delta] + [\beta] = [\delta]$, compute $\alpha \in \mcG_{\sigma(\delta)}$ such that $[\beta] = \connhom_\delta [\alpha]$.
\end{itemize}

\begin{remark}
These algorithms are all the effective versions of certain natural surjections (i.e.\ computable weak sections of these surjections):
\begin{itemize}
\item
	the restriction $s \colon D \to t^{-1}(0)$,
\item
	the collection, indexed by $\delta$, of the action maps $H \to s^{-1}(s[\delta])$, $h \mapsto [\delta]+h$,
\item
	the collection, indexed by $\delta$, of the maps $\connhom_\delta \colon G_{\sigma(\delta)} \to \St [\delta]$ to the stabilizer group of $[\delta]$ under the action of $H$.
\end{itemize}
\end{remark}

Assuming that all terms are weakly locally effective Abelian heaps and a basepoint $[\delta] \in D$ is computable, we will obtain a weakly locally effective exact sequence of Abelian groups in the sense of Section~\ref{s:effective_exactness_groups}, details are given in the proof of Theorem~\ref{thm:inductive_step}.

\subsection{Running times and parametrized effectivity}\label{s:paramobjects}

We will now comment on our approach to the computational complexity of algorithms, a somewhat simplified version of~\cite{polypost} and~\cite{aslep}. We will explain this on the algorithm $H_\then$ computing, for a given finite simplicial set $X$, its $\then$-th homology group $H_\then X$. Of course, this algorithm is quite simple -- setup the chain complex of $X$ and compute its homology using the Smith normal form. The main point, however, is that the algorithm uses as subroutines the algorithms computing the faces of $X$ and also other algorithms of $X$ (its effective structure); otherwise, the algorithm is exactly the same for all simplicial sets. Therefore, the running time of $H_\then X$ depends (heavily) on the running times of the algorithms of $X$ and as such should be treated as a function of these running times. We have decided not to formalize this approach and, instead, we formulate our statements in the following vein: If all the algorithms of $X$ run in polynomial time then so do the algorithms of $H_\then X$ -- here, the algorithms for $X$ are those of an effective simplicial set, while those for $H_\then X$ are those of an effective Abelian group; this involves, in particular, the algorithm outputting the isomorphism type.

Clearly, when speaking about the complexity of computing $H_\then X$, we must consider a class of simplicial sets and $X$ should then be treated as an argument of $H_\then$. Technically, we consider $X$ to be parametrized by a parameter $p \in \sfP$ that involves all the data needed for the computations inside $X(p)$: a number of options is possible, e.g.\ one can specify a finite simplicial complex by the collection of its maximal simplices, one can specify a finite simplicial set as in Section~\ref{s:math_simplicial_sets} by the collection of its non-degenerate simplices and their faces, one can specify the Eilenberg-MacLane space $\Kpin$ via the group $\pi$ (the list of the orders of its cyclic summands), etc. The face operator $d_j$ in all the simplicial sets $X(p)$ of a given class will be required to be computed by a single algorithm that takes $p$ as an extra argument, i.e.\ $d_j(p, x)$ computes the $j$-th face of the simplex $x \in X(p)$. We will then say that $X$ is a \emph{family} of simplicial sets. We thus have a family of finite simplicial complexes, a family of finite simplicial sets, a family of Eilenberg-MacLane spaces $\Kpin$ (for all finitely generated Abelian groups $\pi$), etc.


\begin{definition}
Generally, our computational objects consist of a mathematical object and a set of algorithms. Denoting by $\mcC$ the \emph{class} of all such computational objects, the individual objects will be referred to as \emph{$\mcC$-objects} and the required algorithms as \emph{$\mcC$-algorithms}.
\end{definition}

Thus, to give a $\mcC$-object (e.g.\ a weakly locally effective Abelian group), we need to supply a full set of $\mcC$-algortihms for it (e.g.\ addition, etc.). As explained above, we have an obvious notion of a family of $\mcC$-objects, where an extra parameter $p \in \sfP$ is added.

\begin{definition}
We say that a given $\mcC$-object is \emph{polynomial time} if all the involved $\mcC$-algorithms have polynomial running time. Similarly, there is a notion of a \emph{polynomial time} family of $\mcC$-objects.
\end{definition}

%


\begin{definition} \label{d:construction}
A (computable) \emph{construction} $F \colon \mcC \to \mcD$ is a mapping on the level of mathematical objects (generally multi-valued) together with a full set of $\mcD$-algorithms that are allowed to use formal calls to $\mcC$-algorithms. In this way, a $\mcC$-object $X$ gives rise to a $\mcD$-object $F_*X$, by replacing the formal calls by calls to the actual $\mcC$-algorithms of $X$.

We say that this construction is \emph{polynomial time}, if it preserves polynomial time objects, i.e.\ $X$ polynomial time $\Rightarrow$ $F_*X$ polynomial time.
\end{definition}

Thus, a construction is like a $\mcD$-object modulo $\mcC$-algorithms and as such is suited for studying running times recursively. The following proposition, while very simple to prove, summarizes our approach to the running time analysis and explains why we do not have to deal with \emph{families} of objects explicitly.

\begin{proposition} \label{prop:propagating_polynomial_time_families}
When $X$ is a family of $\mcC$-objects, $F_*X$ is a family of $\mcD$-objects. Assuming $F$ to be a polynomial time construction, if $X$ is a polynomial time family, then so is $F_*X$.
\end{proposition}

With these notions at hand, we may simply say that the kernel is a polynomial time construction. However, we note that the name ``kernel'' only describes the mathematical part of the objects, thus we have to further specify the involved algorithms, i.e.\ that it takes a computable homomorphism between effective Abelian groups and gives an effective Abelian group.

\begin{remark}
In this way, for different families of homomorphisms, we get different families of kernels, i.e.\ the actual codes will differ, although only in the involved calls, so that this approach leads to code duplication. Of course, there are standard ways of dealing with this problem and we will not comment on this further.
\end{remark}

As a corrolary, the $\then$-th homology group $C \mapsto H_n(C)$ is a polynomial time construction (from effective chain complexes to effective Abelian groups) and so is the total homology $C \mapsto H_*(C)$ valued in effective graded Abelian groups. However, the running times of our algorithms are very sensitive to dimension, so that we will need to truncate all objects.

Similarly, the association
\[(0 \to A \to B \to C \to 0) \quad \longmapsto \quad B\]
from weakly locally effective exact sequences with the outer groups effective to effective Abelian groups is a polynomial time construction.

The preimage is interpreted as an association
\[(f \colon A \to B,\, b \in B) \quad \longmapsto \quad a\]
that is multi-valued, with the possibility of having no value at all (if $b$ does not lie in the image of $f$). As such, it is again a polynomial time construction.

\begin{remark}
In a family of weakly locally effective collections of groups and also in a family of computable collections of group homomorphisms, there are then employed two levels of parameters -- one coming from it being a family and the other from it being a collection. We will explain this on the example of the exact sequence~\eqref{eq:general_exact_sequence}: The map $\widetilde\connhom$ takes as arguments pairs $(\elln, h)$ with $\elln \colon X \to \Pnew$ and $h \colon \interval \times X \to \Pold$ a homotopy $\pn\elln \sim \pn\elln$, relative to $A$; note that the conditions on $h$ depend on the parameter $\elln$. In the setting of families, $\widetilde\connhom$ takes an extra argument $p$ that specifies $X$, $Y$, etc.: clearly, the conditions on $\elln$ depend on this extra parameter, so that the arguments dependences are $(p, \elln(p), h(p, \elln(p)))$ and the two parameter levels are not independent.
\end{remark}

\section{Proof of the main theorem}\label{s:main_theorem}
Now that we have defined the main notions, we are able to give a more detailed outline of the proof together with references to the appropriate statements forming the steps in the proof.

Theorem~\ref{thm:stable_situation_heaps} reduces the computation of $[X, Y]^A$ to that of the Abelian heap $[\Sigma X, \Sigma Y]^{\Sigma A}$. Then, according to Theorem~\ref{t:from_Y_to_Pn}, we may replace $\Sigma Y$ by its Postnikov stage $\Pnewst$, as long as $\then \geq \dim (\Sigma X) = 1 + \dim X$. We will thus compute
\[[\Sigma X, \Pnewst]^{\Sigma A}\]
inductively and finish with $\then = 1 + \dim X$. According to Theorem~\ref{t:representing_homotopy_classes}, we further replace $\Pnewst$ by the truncated tower $\Pnew$, so that we are left to compute inductively
\[[\Sigma X, \Pnew]^{\Sigma A}\]
(this, generally non-Abelian, heap is in fact computable even without the stability assumption, using methods of this paper together with those of \cite{FilVokri}).

A pointwise homologically effective Postnikov system is constructed in Theorem~\ref{t:postnikov-restat}. Theorem~\ref{thm:weakly_locally_effective_heap_structure} then equips each $[\Sigma X, \Pnew]^{\Sigma A}$ with a weakly locally effective Abelian heap structure and these are organized in an exact sequence~\eqref{eq:general_exact_sequence}, which is weakly locally effective by Theorem~\ref{thm:exact_sequence} and by induction. Finally, the algorithm of Theorem~\ref{thm:inductive_step} either finds out that $[\Sigma X, \Pnew]^{\Sigma A}$ is empty or equips this set with a structure of an effective Abelian group.

To set up the full computational strength of the main exact sequence, we will need to lift maps along the stages of the Postnikov tower. This is classically handled by the obstruction theory and we will thus have to develop its computational version. The importance of the suspension as a domain will come into the play at the very end; for this reason, we will work all the time with homotopy classes of maps from $X$ to a Postnikov tower $P$ of $Y$ but with a view of applying the machinery to $\Sigma X$, $\Sigma Y$ as explained above.

\subsection*{Running times}

All the results invoked in the above proof are described as polynomial time constructions of Definition~\ref{d:construction}. Thus, so is their composition
\[(A \subseteq X, Y, f \colon A \to Y) \mapsto [X,Y]^A\]
landing in the class of effective Abelian groups. Proposition~\ref{prop:propagating_polynomial_time_families} then shows that any polynomial time family of the input data produces a polynomial time family of the outputs. This holds regardless of an explicit way of encoding the input diagrams and maps, it is only required that all the (obvious) tasks should be computable in polynomial time; on the other hand, we outlined in Section~\ref{s:paramobjects} a simple way of encoding finite simplicial sets and this can be easily extended to finite diagrams of finite simplicial sets. Through this encoding (or any other), we thus obtain a concrete realization of our algorithm that is polynomial time.

\section{Computational obstruction theory}\label{s:obstruction theory}

\begin{assumption} \label{a:PT}
We will assume throughout this section that the Postnikov system $\varphi \colon Y \to P$ is pointwise homologically effective; later, this will be strengthened but at this point we will make do with $P$ being locally effective and all diagrams of homotopy groups being effective.

In addition, we will assume that $(X,A)$ is an effective pair of diagrams.

For running time analysis, we have to fix a bound on the dimensions of all objects. In this case, one should read all of the statements below as describing polynomial time constructions. That is, given that the pointwise homologically effective Postnikov system is polynomial time, as well as all the additional input data in the statements, the same holds for the output.
\end{assumption}

We will now describe (effective) obstruction theory for diagrams: in order to lift a homotopy class in $[X, \Pold]^A$ to a homotopy class in $[X,\Pnew]^A$, we represent the original homotopy class by a map of towers $X \to \Pold$ under $A$ so that we get, according to Theorem~\ref{t:representing_homotopy_classes}, to the following equivalent situation.

\begin{proposition} \label{prop:lift_ext_one_stage}
Under Assumpition~\ref{a:PT}, there is an algorithm that, given a computable commutative square of towers
\[\xymatrix@C=30pt{
A \ar[r] \ar@{ >->}[d] & \Pnew \ar@{->>}[d]^\pn \\
X \ar[r] \ar@{-->}[ru] & \Pold
}\]
decides whether an indicated lift exists. If it does, it computes one. If $H^{n+1}(X,A;\pi)=0$, then such a lift is guaranteed to exist.
\end{proposition}

%

\begin{proof}
Since $\pn$ is a pullback of $\deltan$, we obtain an equivalent lifting problem
\[\xymatrix{
	A \ar[r] \ar[d] & \Pnew \ar[r] \ar[d] \pb & \En \ar[d] \\
	X \ar[r] \ar@{-->}[urr] & \Pold \ar[r] & \Kn
}\]
By adjunction between $\towercomp{}{\then}$ and $\susp{}{\then}$, this is further equivalent to the corresponding lifting problem at level $\then$:
\[\xymatrix{
	\An \ar[r] \ar[d] & \Pnewst \ar[r] \ar[d] \pb & \Enst \ar[d] \\
	\Xn \ar[r] \ar@{-->}[urr] & \cofrst{\Poldst} \ar[r] & \Knst
}\]
This lifting problem is translated to a cohomological problem in $C^{n+1}(\Xn, \An; \pi_n)$ as usual -- Lemma~\ref{lem:lift_ext_one_stage} makes this translation and solves the cohomological problem.

Finally, since $\Pnewst$ is a pullback in the right hand side square by Lemma~\ref{lem:Postnikov_tower_via_levels}, it is easy to compute a lift $\Xn \to \Pnewst$, and thus $X \to \Pnew$, from the lift in the composite square. 
\end{proof}

\begin{lemma} \label{lem:lift_ext_one_stage}
There is an algorithm that, given an effective pair of diagrams $(X',A')$, an effective diagram of Abelian groups $\pi \in \AbI$ and a computable commutative square of diagrams
\[\xymatrix@C=30pt{
A' \ar[r]^-c \ar@{ >->}[d] & \Engen \ar@{->>}[d]^\delta \\
X' \ar[r]_-z \ar@{-->}[ru] & \Kngen
}\]
decides whether an indicated lift exists. If it does, it computes one. If $H^{n+1}(X',A';\pi)=0$, then a lift exists for every $c$ and $z$.
\end{lemma}

\begin{proof}
Thinking of $c$ as a cochain in $C^n(A';\pi)$, we extend it to a cochain on $X'$ by mapping all $n$-cells not in $A'$ to zero. This prescribes a map $\widetilde c \colon X' \ra \Engen$ that is a solution of the lifting-extension problem from the statement for $z$ replaced by $\delta\widetilde c$. Since the lifting-extension problems and their solutions are additive, one may subtract this solution from the previous problem and obtain an equivalent lifting-extension problem
\[\xymatrix@C=30pt{
A' \ar[r]^-{0} \ar@{ >->}[d] & \Engen \ar@{->>}[d]^\delta \\
X' \ar[r]_-{z-\delta\widetilde c} \ar@{-->}[ru]^-{c_0} & \Kngen
}\]
A solution of this problem is a relative cochain $c_0$ whose coboundary is $z_0=z-\delta\widetilde c$ (this $c_0$ yields a solution $\widetilde c+c_0$ of the original problem). Since $C_*(X',A')$ is effective, such a $c_0$ is computable whenever it exists (and it always exists in the case $H^{n+1}(X',A';\pi)=0$).
%
\end{proof}

Whenever $C_*(X, A)$ is acyclic, there exists a contraction of $C_*(X, A)$ (see Proposition~\ref{prop:gen} and recall that $(X, A)$ is assumed to be effective) and therefore its cohomology groups with arbitrary coefficients are zero. Thus, in this situation, all possible obstructions are zero and we may proceed inductively, using Proposition~\ref{prop:lift_ext_one_stage}, to lift through arbitrary number of stages. As special cases, we obtain the following two results.

\begin{proposition}[homotopy lifting] \label{prop:homotopy_lifting}
Under Assumpition~\ref{a:PT}, given a computable commutative square
\[\xymatrix{
(\vertex 0 \times X)\cup(\interval\times A) \ar[r] \ar@{ >->}[d]_-\sim & \Pnew \ar@{->>}[d] \\
\interval \times X  \ar[r] \ar@{-->}[ru] & \Polder
}\]
it is possible to compute a lift. In other words, one may lift homotopies in Postnikov towers algorithmically.
\end{proposition}

\begin{proof}
The chain complex $C_*(\interval\times X, (\vertex 0\times X)\cup(\interval\times A)) \simeq C_*(\interval, 0) \otimes C_*(X, A)$ is acyclic, since $C_*(\interval, 0)$ is.
\end{proof}

The second result concerns algorithmic concatenation of homotopies.
Let $\horn21$ denote the first horn in the standard $2$-simplex $\stdsimp2$, i.e.\ the simplicial subset of the standard simplex $\stdsimp2$
spanned by the faces $\face{01}$ and $\face{12}$, where $\face{jk} \cong \interval$ denotes the subsimplex with vertices $j$, $k$.
Given two homotopies $h_2,h_0\colon\interval\times X\to Y$
that are compatible, in the sense that $h_2$ is a homotopy from $\ell_0$ to $\ell_1$ and $h_0$ is a homotopy from $\ell_1$ to $\ell_2$,
one may prescribe a map $\horn21\times X\to Y$ as $h_2$ on
$\face{01} \times X$ and as $h_0$ on $\face{12} \times X$.
If this map has an extension
$H\colon \stdsimp2\times X\to Y$, then the restriction of $H$ to
 $\face{02} \times X$ gives a homotopy from $\ell_0$ to $\ell_2$,
which can be thought of as a \emph{concatenation} of $h_2$ and $h_0$.
We will need the following effective,
relative version; the proof is entirely analogous
to that of the previous proposition and we omit it.

\begin{proposition}[homotopy concatenation] \label{prop:homotopy_concatenation}
Under Assumpition~\ref{a:PT}, given a computable commutative square
\[\xymatrix{
(\horn{2}{i}\times X)\cup(\stdsimp{2}\times A) \ar[r] \ar@{ >->}[d]_-\sim & \Pnew \ar@{->>}[d] \\
\stdsimp{2}\times X  \ar[r] \ar@{-->}[ru] & \Polder
}\]
it is possible to compute a lift. In other words, one may concatenate homotopies in Postnikov towers algorithmically.\qed
\end{proposition}

Since the algorithm takes the restrictions of homotopies to $(\partial \interval \times X) \cup (\interval \times A)$ as an input, we obtain as a corollary:

\begin{corollary} \label{c:wle_htpy_classes}
Under Assumpition~\ref{a:PT}, $[\interval \times X, \Pold]^A_\bullet$ is a weakly locally effective collection of groups.
\end{corollary}

In the case that the homotopies are not relative, i.e.\ constant on $A$, it is not possible to concatenate and we only get a heap structure.


\begin{theorem} \label{thm:weakly_locally_effective_heap_structure}
Under Assumpition~\ref{a:PT}, it is possible to equip $[\Sigma X, \Pnew]^{\Sigma A}$ with a weakly locally effective Abelian heap structure.
\end{theorem}

\begin{proof}
There is an obvious isomorphism coming from the definition of a suspension as a quotient of the cylinder:
\[[\Sigma X, \Pnew]^{\Sigma A} \cong [\interval \times X, \Pnew]^{(\partial \interval \times X) \cup (\interval \times A)}\]
(in fact, the right hand side is slightly more general in that it allows the maps to be fixed on the two ends of the cylinder in a non-constant way). In accordance with Theorem~\ref{t:representing_homotopy_classes}, we will work with maps of towers $\interval \times X \to \Pnew$ under the appropriate subspaces. Given three maps of towers $\ell_1 \comma o \comma \ell_2 \colon \interval \times X \to \Pnew$, we organize them into a single map
\[(\face{13}\times X)\cup(\face{12}\times X)\cup(\face{02}\times X)\xlra{(\ell_1,o,\ell_2)}\Pnew\]
and we note that the tower on the left is $(\face{13}\cup \face{12}\cup \face{02})\times X$. 
\[\xy
<0pc,0pc>;<2pc,0pc>:
(-2,1)*+{\scriptstyle\bullet}="a0"*++!R{0},"a0";
(0,2)*+{\scriptstyle\bullet}="a2"*++!D{2},"a2"**\dir{-};?>*\dir{>},?<>(.5)*+!D{\scriptstyle\ell_2},
(0,0)*+{\scriptstyle\bullet}="a1"*++!U{1},"a1"**\dir{-};?<*\dir{<},?<>(.3)*+!L{\scriptstyle o},
(2,1)*+{\scriptstyle\bullet}="a3"*++!L{3},"a3"**\dir{-};?>*\dir{>},?<>(.5)*+!U{\scriptstyle\ell_1},
"a0";"a3"**\dir{},?<>(.35)*+!U{\scriptstyle\ell},?<>(.45)**\dir{--},"a0";"a3"**\dir{};?<>(.55);**\dir{--},?>*\dir{>},
"a0";"a1"**\dir{.},"a2";"a3"**\dir{.}
\endxy\]
Together with the composite
\[\stdsimp 3\times A\xlra{s^0s^2 \times \id} \interval \times A \lra \Pnew\]
(the first map takes vertices $0, 1 \in \stdsimp 3$ to the vertex $0 \in \interval$ and vertices $2, 3 \in \stdsimp 3$ to $1 \in \interval$, while the second map is the common restriction of $\ell_1,o,\ell_2$), these describe the top map in the diagram
\[\xymatrix@C=3pc{
((\face{13}\cup \face{12}\cup \face{02})\times X)\cup(\stdsimp 3\times A) \ar[r] \ar@{ >->}[d]_-\sim &
	\Pnew \\
\stdsimp{3}\times X \ar@{-->}[ru]
}\]
An extension can be computed inductively using Proposition~\ref{prop:lift_ext_one_stage} as in the previous special cases and its restriction to $\face{03}\times X$, denoted by $\ell$ in the above picture, gives a representative of $[\ell_1] +_{[o]} [\ell_2]$. It is standard that the resulting map $\ell$ is unique up to homotopy relative to $(\partial \interval \times X) \cup (\interval \times A)$.
\end{proof}

We remark that, as a slight simplification, it is enough to extend first to the face $\face{123} \times X$ and then to $\face{023} \times X$ in order to obtain $\ell$.

\section{An exact sequence and the inductive computation}\label{s:exact_sequence}

\begin{assumption} \label{a:inductive}
In addition to Assumption~\ref{a:PT}, we require the inductive hypothesis that $[\interval \times X,\Pm]^\partial_\bullet$ is an effective collection of Abelian groups, for all $\them < \then$. In fact, this holds even without the stability assumption, although these groups are generally not Abelian but only polycyclic, see~\cite{FilVokri}.

Again, once a bound on the dimensions of all objects is fixed, all the theorems prescribe polynomial time constructions, i.e.\ if all the inputs are assumed to be polynomial time, the same holds for the output.
\end{assumption}

We recall the exact sequence~\eqref{eq:general_exact_sequence}:
\[[\interval \times X,\Pold]^\partial_\bullet \xlra{\connhom_\bullet} [X,\Ln]^A \acts [X,\Pnew]^A \xlra{\pnst} [X,\Pold]^A \xlra{\knst} [X,\Kn]^A,\]
%

\begin{theorem} \label{thm:exact_sequence}
Under Assumpition~\ref{a:inductive}, the above is a weakly locally effective exact sequence.
\end{theorem}

\begin{proof}
All terms are represented by maps of towers. This is covered by Theorem~\ref{t:representing_homotopy_classes} and Lemma~\ref{l:representing_homotopy_classes_EM}. The computability of various maps in the diagram is clear and the weak locally effective group structure on the first term is provided by Corollary~\ref{c:wle_htpy_classes}. The computability of the basepoint of the last term follows from Lemma~\ref{lem:En_contractible}, since the basepoint is given as an arbitrary composition $X \to \En \xra{\deltan} \Kn$.

A weak section of $\pnst$ is computed by lifting a map $\ellnmo \colon X \to \Pold$ to $\Pnew$, using Proposition~\ref{prop:lift_ext_one_stage}. A weak section of $\connhom_\bullet$ is computed as follows: given $\elln \colon X \to \Pnew$ and $\zeta \colon X \to \Ln$ such that $\elln \sim \elln + \zeta$, compute such a homotopy $h$ using Lemma~\ref{l:homotopies_computable} and project it to $\Pold$ to obtain a preimage $\pn h$, according to $\connhom_{[\elln]}[\pn h] = [\zeta]$. A weak section for the action is computed by a combination of the previous ingredients: given $\elln$, $\elln'$ such that $\pn \elln \sim \pn \elln'$, compute such a homotopy $h$ using Lemma~\ref{l:homotopies_computable} and lift it, using Proposition~\ref{prop:homotopy_lifting} to a homotopy $\elln  + \zeta \sim \elln'$ to obtain a preimage $\zeta$, according to $[\elln] + [\zeta] = [\elln']$.
\end{proof}

\begin{lemma} \label{lem:En_contractible}
For any effective pair $(X, A)$ of diagrams and any computable map $A \to \En$, we have $[X, \En]^A = *$ and a representative can be computed algorithmically.
\end{lemma}

\begin{proof}
We have $[X, \En]^A \cong [\towercomp{X}{n}, \Enst]^{\towercomp{A}{n}}$. For any cellular $A' \to X'$ the extension problem
\[\xymatrix{
A' \ar[r] \ar[d] & \Engen \\
X' \ar@{-->}[ru]
}\]
is solvable -- it just means that any cochain in $C^n(A'; \pi)$ extends to a cochain in $C^n(X'; \pi)$; any such extension is determined by the images of the cells of $X' \smallsetminus A'$, e.g.\ we may assign them the zero value.

Applying this to $\towercomp{A}{n} \to \towercomp{X}{n}$ and to $(\partial\Delta^1 \times \towercomp{X}{n}) \cup (\Delta^1 \times \towercomp{A}{n}) \to \Delta^1 \times \towercomp{X}{n}$ yields the existence of a map and the existence of a relative homotopy between any two such maps.
\end{proof}


\begin{lemma} \label{l:homotopies_computable}
An algorithms exists, computing for any given representatives $\elln$, $\elln'$ with $[\elln] = [\elln'] \in [X, \Pnew]^A$ a homotopy $\elln \sim \elln'$ relative to $A$.
\end{lemma}

\begin{proof}
Given $\elln \colon X \to \Pnew$, we define $\ell_\them$ to be the composition of $\elln$ with the canonical projection $\Pnew \to \Pm$. We proceed by induction on the height $\them$ of the Postnikov tower to compute a homotopy $h_{\them} \colon \ell_\them \sim \ell_{\them}'$. When a homotopy $h_{\them-1} \colon \ell_{\them-1} \sim \ell_{\them-1}'$ has been computed, we lift it, using Proposition~\ref{prop:homotopy_lifting}, to a homotopy $\widetilde h_{\them-1}\colon\ell''_\them\sim\ell_\them'$ from some map $\ell''_\them$ lying over $\ell_{\them-1}$. Then $\ell''_\them = \ell_\them + \zeta_\them$ for a unique $\zeta_\them \colon X \to \Lm$, namely $\zeta_\them = \ell''_\them - \ell_\them$.

Since Proposition~\ref{prop:homotopy_concatenation} provides algorithmic means for concatenating homotopies, it remains to construct a homotopy $h'_\them \colon \ell_\them \sim \ell''_\them$. Consider the connecting homomorphism in~\eqref{eq:general_exact_sequence} for stages $\Pmmo$ and $\Pm$, i.e.
\[\connhom_{\ell_\them}\colon[\interval \times X,\Pmmo]^\partial \lra[X,\Lm]^A.\]
From the exactness of~\eqref{eq:general_exact_sequence} and from $\ell_\them \sim \ell_\them' \sim \ell''_\them = \ell_\them + \zeta_\them$, it follows that $[\zeta_\them]$ lies in the image of $\connhom_{\ell_\them}$. Since the target is effective by Lemma~\ref{lem:fully_eff_cohlgy}, the algorithm of Lemma~\ref{l:preimage} then computes some $h_{\them-1}'$ such that $\connhom_{\ell_\them}[h_{\them-1}']=[\zeta_\them]$.

It is then easy to see (cf.\ \cite[Proposition~7]{FilVokri}) that the required homotopy $h'_\them \colon \ell_\them \sim \ell''_\them$ can be computed as a lift of the homotopy $h_{\them-1}'$ as in (i.e.\ both ends of $h_\them'$ prescribed):
\[\xymatrix{
(\partial\interval \times X)\cup(\interval \times A) \ar[r] \ar[d] & \Pm \ar[d] \\
\interval \times X \ar[r]_-{h_{\them-1}'} \ar@{-->}[ru]_-{h'_\them} & \Pmmo
}\]
Proposition~\ref{prop:lift_ext_one_stage} provides an algorithm for the computation of $h_\them'$ and the proof is finished.
\end{proof}

\begin{theorem}\label{thm:inductive_step}
Under Assumpition~\ref{a:inductive}, if the main exact sequence consists of weakly locally effective Abelian heaps and heap homomorphisms and if $[X, \Pold]^A$ is equipped with an effective Abelian group structure, then it is possible to algorithmically decide whether $[X, \Pnew]^A$ is non-empty and, if this is the case, further equip $[X, \Pnew]^A$ with an effective Abelian group.
\end{theorem}

\begin{remark} \label{rk:H_space_approach}
Our proof of the main theorem replaces the computation of $[X, Y]^A$ by that of $[\Sigma X, \Pnewind{\Sigma Y}]^{\Sigma A}$, for $\then \geq 1 + \dim X$, with $\Pnewind{\Sigma Y}$ the Postnikov stage of $\Sigma Y$, and then utilizes the above theorem, since the latter carries a natural weakly locally effective heap structure by Theorem~\ref{thm:weakly_locally_effective_heap_structure}. It is also possible to put a weakly locally effective heap structure directly on $[X, \Pnewind Y]^A$, with $\Pnewind Y$ the Postnikov stage of $Y$; namely, it is possible to construct a weak H-space structure (or rather a heap version of an H-space structure) on the Postnikov stage $\Pind Y$, under the assumption $\then \leq 2 \conn Y$. The advantage of this approach is that we need only $\dim X$ stages of the Postnikov tower $\Pnewind Y$ and we believe that this should make the resulting algorithm faster in practice. We have decided to use the suspension $\Sigma Y$ mainly for the simplicity of the heap operation.
\end{remark}

We now proceed with a few preliminary results needed for the proof of Theorem~\ref{thm:inductive_step}.

\subsection{Translating zero}

Let $S$ be a group and $p \in S$ an element. We define a new group structure on $S$, denoted by $S_p$, by declaring the right translation $S \to S_p$, $x \mapsto x + p$, to be an isomorphism. Consequently, $S_p$ has zero $p$, addition $x +_p y = x - p + y$ and inverse $-_p x = p - x + p$.

\begin{proposition}
If the group $S$ is weakly locally effective or effective, then so is $S_p$.
\end{proposition}

\begin{proof}
The weak local effectiveness is obvious from the formulas. If $S$ is effective with generators $a_i$ of orders $q_i$, then $S_p$ is effective with generators the translates $a_i + p$ of the same orders $q_i$. An expression of $a$ as an integral combination in $S_p$ is obtained by translating to $S$, i.e.\ by computing the coefficients of $a - p$ as an integral combination of the $a_i$ in $S$.
\end{proof}

%
%
%
%

Before starting the proof of Theorem~\ref{thm:inductive_step}, we prove the full effectivity of the cohomology groups of an effective pair $(X, A)$. This will be the basic building stone.

\begin{lemma}\label{lem:fully_eff_cohlgy}
Let $(X,A)$ be an effective pair of diagrams. Let $c \colon A \to \Engen$ be a fixed computable map and make $\Kngen$ into a diagram under $A$ via $\delta c$. Then it is possible to equip $[X,\Kngen]^A$ with a structure of an effective Abelian group; the elements are represented by maps $X \to \Kngen	$ whose restriction to $A$ equals $\delta c$.
\end{lemma}

\begin{proof}
According to Lemma~\ref{lem:En_contractible}, the set $[X, \Engen]^A$ has a single element, obtained by extending $c$ to a map $\widetilde c \colon X \to E(\pi, n)$ and, thus, there is a well defined element $[\delta \widetilde c] \in [X, \Kngen]^A$; it will serve as the zero of the group. Denoting the group from the statement temporarily by $[X, \Kngen]^{A, c}$ to stress the chosen map $c$, we have an isomorphism
\[[X, \Kngen]^{A, c} \cong [X, \Kngen]^{A, 0} ,\, [z] \mapsto [z - \delta \widetilde c],\]computable together with its inverse. We will thus assume from now on that $c=0$ and drop it again from the notation.

In this situation we have an isomorphism
\[[X, \Kngen]^A \cong H^{n+1}(X,A;\pi)\]
computable in both directions. Since the cochain complex
\[C^*(X,A;\pi)=\operatorname{Hom}(C_*(X,A),\pi)\]
clearly consists of effective Abelian groups and since these are closed under subgroups and quotients, the cohomology group is also effective.
\end{proof}

The proof of Theorem~\ref{thm:inductive_step} consists of two main steps:

\subsection{Computing the basepoint of $\boldsymbol{[X, \Pnew]^{A}}$}

Since the group $[X, \Pold]^A$ is effective, it is equipped with a zero $[o_{n-1}]$. We first solve the problem of choosing a zero $[o_n] \in [X, \Pnew]^A$.
\[\xymatrix{
A \ar[r]^-{\fn} \ar[d]_-{i} & \Pnew \ar[r]^-{\qn} \ar[d]^-{\pn} & \En \ar[d]^-{\deltan} \\
X \ar[r]_-{\onmo} & \Pold \ar[r]_-{\kn} & \Kn
}\]
Considering the set $[X, \Kn]^A$ of homotopy classes of maps whose restriction to $A$ equals $\deltan \qn \fn$, and equipping it with zero $[\kn\onmo]$, the map $\knst \colon [X, \Pold]^A \to [X, \Kn]^A$ becomes a computable homomorphism between effective Abelian groups. We remark that the zero $[\kn\onmo]$ is generally different from the natural zero exhibited in the proof of Lemma~\ref{lem:fully_eff_cohlgy} -- we will denote this natural zero by $0$.

According to Lemma~\ref{l:preimage}, it is possible to decide whether $0$ lies in the image of $\knst$ and compute some $o_{n-1}'$ such that $\knst [o_{n-1}'] = 0$. Then, using Proposition~\ref{prop:lift_ext_one_stage}, it is possible to lift $o_{n-1}' \colon X \to \Pold$ to a map $o_n \colon X \to \Pnew$ that will represent our new basepoint $[o_n] \in [X, \Pnew]^A$. If $0$ does not lie in the image of $\knst$ then $[X, \Pnew]^A$ is empty.

\subsection{Making $\boldsymbol{[X, \Pnew]^{A}}$ effective} \label{s:sections}

Having computed $[o_n]$, our general exact sequence~\eqref{eq:general_exact_sequence} becomes, under our assumptions, an exact sequence of Abelian groups, that can be easily transformed into a short exact sequence
\[\xymatrix{
0 \ar[r] & \coker\connhom_{\on} \ar[r]_-{\jnst} & [X,\Pnew]^A \ar[r]_-{\pnst} \POS[l]+R*{\vphantom{|}}="a";[]+L*{\vphantom{|}} \ar@<-5pt>@/_2pt/"a"_-\rho & \ker\knst \ar[r] \POS[l]+R*{\vphantom{|}}="a";[]+L*{\vphantom{|}} \ar@<-5pt>@/_2pt/"a"_-\sigma & 0.
}\]
By Lemma~\ref{l:ker_coker}, both $\coker\connhom_\on$ and $\ker\knst$ are effective. Since the indicated weak sections are induced by those of~\eqref{eq:general_exact_sequence}, Lemma~\ref{l:ses} applies and $[X,\Pnew]^A$ becomes effective.

\section{Effective homological algebra}\label{sec:effective}

\heading{Homologically effective diagrams}\label{sec:effdiagrams}
In this section, we define homologically effective diagrams of chain complexes and simplicial sets, introduced originally in the article \cite{Filakovsky2014} under the name of diagrams with effective homology, and describe several constructions with such diagrams. 

We begin by introducing reduction and strong equivalence of diagrams:
\begin{Def}
Let $\dC,\dC' \in \ChI$ be diagrams of chain complexes. A \emph{reduction} $\dC \Redu \dC'$  is a triple of natural transformations  $(\alpha,\beta,\eta)$
\[(\alpha,\beta,\eta) \colon \dC\Redu \dC'\quad\equiv\quad\xymatrix@C=30pt{
\dC \ar@(ul,dl)[]_{\eta} \ar@/^/[r]^\alpha & \dC' \ar@/^/[l]^{\beta}
}\]
such that $\alpha$ and $\beta$ are chain maps satisfying the following conditions:
\begin{equation}\label{eq:reduction}
\begin{array}{lll}
 \eta\beta = 0 & \alpha \eta = 0&  \eta\eta = 0  \\
\alpha \beta  = \id & \bo\eta + \eta \bo = \id - \beta\alpha&
\end{array}
\end{equation}
\end{Def}

One of the most important and well known examples of a reduction is the following, first given in~\cite{EilenbergMacLane:GroupsHPin1-1953,EilenbergMacLane:GroupsHPin2-1954}:
\begin{ex}[Eilenberg--Zilber reduction]\label{p:EZred}
Let $X,Y$ be simplicial sets. Then there is a reduction
\[
C_*(X \times Y) \Redu C_*(X) \otimes C_*(Y)
\]
\end{ex}

The operators in the reduction data can be computed using the acyclic models theorem as e.g.\ in \cite{may}, Chapter~28  and they are not unique. We will further use the reduction data presented in Theorem~2.1a of \cite{EilenbergMacLane:GroupsHPin1-1953}. An important observation is that the operators of the reduction data are based on the face and degeneracy maps which means that the reduction is functorial in simplicial sets, so that the above example extends to diagrams of simplicial sets.

\begin{Def}
A strong equivalence $\dC \steq \dC'$ of diagrams of chain complexes is defined as a span of reductions $\dC\lredu \widehat{\dC}\Redu \dC'$.
\end{Def}
Strong equivalences of diagrams can be composed as in the case of strong equivalences of chain complexes.

Given a category $\icat$, we denote by $\widetilde{\icat}$ the category with the same set of objects but with identity arrows only. There is an obvious inclusion $\widetilde{\icat} \to \icat$ and thus a diagram $X \colon \icat \to \ccat$ induces a diagram $\widetilde X \colon  \widetilde{\icat}\to \ccat$.

The following definition generalizes the concept of an object with effective homology (see~\cite{RubioSergeraert:ConstructiveAlgebraicTopology-2002}) to the context of diagrams:
\begin{Def}\label{def:effective}\hfill
\begin{itemize}

\item We say that a locally effective diagram of chain complexes $\dC$ is \emph{homologically effective} if there is given an effective diagram $C^{\ef} \in \ChI$ and a strong equivalence ${{\dC} \steq C^{\ef}}$.

\item We say that a locally effective diagram of chain complexes $\dC$ is \emph{pointwise homologically effective} if its restriction $\widetilde C \in \ChDiag{\widetilde\mcI}$ is homologically effective. Concretely, this consists of a collection of strong equivalences $C(i) \steq C^{\ef}(i)$ (not necessarily natural in $i \in \mcI$).

\item A locally effective diagram of simplicial sets $\dX \in \Diag$ is \emph{(pointwise) homologically effective} if $\dC_* (\dX)$ is (pointwise) homologically effective.
\end{itemize}
\end{Def}
As an important example, the diagram $WK(\dpi,\thedim)$, where $\dpi \in \AbI$ and $\thedim\ge 1$ is pointwise homologically effective, which follows from \cite[Theorem~3.16]{polypost}. In general, we aren't aware of an algorithm that would show it is homologically effective. However, in special cases such as when $WK(\dpi,\thedim)$ is cofibrant, we obtain homologically effective $WK(\dpi,\thedim)$ (see Proposition~\ref{prop:cofibrep}).
\subsection{Constructions with homologically effective diagrams}
We now introduce the standard results of homological perturbation theory in the context of diagrams of chain complexes. We will utilise them to prove that certain objects, such as cofibrant replacements, are homologically effective.
 
\begin{Def}
Let $C,C'\in \ChI$. Notice that the differential $\bo$ on $C$ can be seen as a natural transformation $C \to C [1]$ satisfying $\bo\bo = 0$. Here $C[1]$ is diagram of chain complexes $C$ with all the chain complexes moved up by one dimension.
We call a collection of maps $\delta : C \to C[1] $ \emph{perturbation} if the sum $\bo + \delta$ is also a differential.
\end{Def}

We now formulate the lemmas.
\begin{lem}[Easy Perturbation Lemma] \label{lem:epl}
Let $(\alpha,\beta,\eta) \col (C, \bo)\Ra  (C',\bo')$ be a reduction of diagrams of chain complexes. Suppose $\delta'$ is a perturbation of the differential $\bo'$. Then there is a reduction $(\alpha,\beta,\eta) \col (C, \bo + \beta \delta' \alpha) \Redu  (C',\bo' + \delta)$.
\end{lem}
\begin{lem}[Basic Perturbation Lemma] \label{lem:bpl}
Let $(\alpha,\beta,\eta) \col (C, \bo)\Redu  (C,\bo')$ be a reduction of diagrams of chain complexes. Suppose $\delta$ is a perturbation of the differential $\bo$ and further  for every $i\in \icat$ and every $c \in C (i)$ there is some $k \in \mathbb{N}$ such that we get  $(\eta\delta)^{k} (c) = 0$. Then there is a reduction of diagrams of chain complexes
 $(\alpha',\beta',\eta') \col (C, \bo + \delta) \Redu  (C' ,\bo' + \delta')$.
\end{lem}
Proof of both lemmas follows directly from the original classical perturbation lemmas (see e.g.\ \cite{RubioSergeraert:ConstructiveAlgebraicTopology-2002}), as there is a concrete desription of new reduction data in terms of sums of compositions of $\alpha, \beta, \eta, \bo, \delta$. As we assume these are natural transformations, the resulting reduction data will consist of natural transformations as well. 
 \begin{lem} \label{lem:product}
 \
\begin{enumerate}
\item Let $C\in\ChI$,  $D\in\ChJ$ be homologically effective. Then so is $C \mathbin{\widehat{\otimes}} D \in \ChDiag{(\mcI \times \mcJ)}$, given by $C \mathbin{\widehat{\otimes}} D (i,j) = C(i) \otimes D(j)$.

In particular, the tensor product of a homologically effective diagram of chain complexes with a homologically effective chain complex is a homologically effective diagram of chain complexes.
\item Let $C,D \in \ChI$ be homologically effective. Then so is $C \oplus D \in \ChI$.
\item Let $C,D \in \ChI$ be homologically effective and $f\colon C \to D$ computable. Then the mapping cylinder $M_f \in \ChI$ is homologically effective.
\end{enumerate}
\end{lem}
\begin{proof}
%
In the first point, the strong equivalences are closed under the tensor product, so that we have $C \mathbin{\widehat\otimes} D \steq C^\ef \mathbin{\widehat\otimes} D^\ef$ and the right hand side $C^\ef \mathbin{\widehat\otimes} D^\ef$ is effective with cells the tensor products of the cells of $C^\ef$ and of $D^\ef$.

The special case in the first point is obtained by taking $\mcJ$ to be the trivial one-object category.

The second point is trivial and the final claim follows from \cite[Proposition~3.8]{polypost}.
\end{proof}

In what follows we are going to apply a general lemma about filtered diagrams of chain complexes. Let  $C \in \ChI$. We consider a filtration $F$ on diagram $C$ of chain complexes: 
\[
0 = F_{-1} C \subseteq F_0 C  \subseteq F_1 C  \subseteq \cdots
\]
such that $C = \bigcup_k F_k C$. We further assume that each $F_k C$ is a cellular subcomplex i.e.\ it is generated by a subset of the given basis of $C$ and that the filtration is locally finite i.e.\ for each $n$ we have $F_k C_n = C_n$ for $k \gg 0$.

\begin{lem}[\cite{aslep}, Lemma 7.3] \label{lem:main}
Let $C \in \ChI$ be a diagram of chain complexes with filtration $F$ satisfying properties as above. If each filtration quotient $G_k C = F_k C / F_{k-1} C$ is homologically effective then so is $C$.
\end{lem}
\begin{proof}
We define $G = \bigoplus\limits_{k \geq 0} G_k$. The sum is \emph{locally finite}: By the properties of $F$, we get that $G_k C_n = 0$ for $k \gg 0$. Thus for each $n$, we get a direct sum of homologically effective diagrams $G_k C_n \in \ChI$ and it follows that $G$ is homologically effective.

The diagram $C$ can be seen as a perturbation of $G$. This perturbation decreases the filtration degree. If we take a direct sum of given strong equivalences ${G_k \lredu \widehat{G_k}\Redu G_k ^\ef}$, we obtain a strong equivalence $G  \lredu \widehat{G}\Redu G ^\ef$. All the chain complexes are equipped with a filtration degree. Since the perturbation on $G$ decreases the filtration degree, while the homotopy operator preserves it, we can apply the perturbation 
lemmas \ref{lem:epl}, \ref{lem:bpl} to obtain a strong equivalence $C  \lredu \widehat{C}\Redu C ^\ef$. \end{proof}

The main application of the lemma is in the proof of the following result:
\begin{proposition}[Theorem 1.3, Proposition 1.2 in \cite{Filakovsky2014}]\label{prop:cofibrep} \noindent
Let $X \in \Diag$ be a pointwise homologically effective diagram. Then its Bousfield-Kan cofibrant replacement $X^\cofr$ is homologically effective.
\end{proposition}
\begin{proof}
We remind that for any category $\icat$ there is a simplicial set $N \icat$, the nerve of $\icat$. The simplicial set $N \icat$ can be seen as a homotopy colimit of the diagram consisting of points. Then there is a projection  $q \colon X^\cofr \to N \icat$ given as a projection onto
\[
\bigsqcup_{n\geq 0 \quad i_0, \cdots , i_n} \stdsimp{n} \times \icat (i_1, i_0) \times \cdots \times \icat (i_n,i_{n-1})/ {\sim}
\]
and we define the skeleton of $ X^\cofr$:
\[
\sk_{k} X^\cofr = q^{-1} (sk_{k} N \icat).
\]

We want to use Lemma~\ref{lem:main} to prove that the diagram $C_*(X^\cofr) \in \ChI$ is homologically effective. Therefore, we first have to introduce a filtration $F$ on the diagram of chain complexes $C_*(X^\cofr)$. We define $F$ as follows:
\[
F_k C_* (X^\cofr) = C_* ( \sk_{k} X^\cofr)
\]
Denoting $G_k = F_k / F_{k-1}$, we get
\[
G_k (C_* (X^\cofr)) = \bigoplus\limits_{\substack{i_0\la \cdots \la i_k \\ \text{ nondeg.} }}  C_* (\stdsimp{k} \times X(i_0)  \times \icat(-, i_k) , \bo \stdsimp{k} \times X(i_0)  \times \icat(-, i_k)) . \\
\]
The sum is taken over chains of morphisms in $\icat$ that do not contain identity as those are cancelled out when taking the quotient $G_k = F_k/ F_{k-1}$. By the finiteness of $\icat$, the number of nondegenerate chains of morphisms of length $k$ is finite, so the sum is finite.

The Eilenberg--Zilber reduction yields in this case
\begin{equation} \label{eq:formula}
 G_k (C_* (X^\cofr)) \Redu \bigoplus\limits_{\substack{i_0\la \cdots \la i_k \\ \text{ nondeg.}}}  C_* (\stdsimp{k}, \bo \stdsimp{k}) \otimes C_*(X(i_0))  \otimes \mathbb{Z}\icat (-, i_k)
\end{equation}
By definition, $\mathbb{Z}\icat (-, i_k) \in \ChI$ is effective and so is $C_*(\stdsimp{k}, \bo \stdsimp{k}) \in \Ch$ with a single generator in dimension $k$. By assumption, $C_*(X(i_0)) \in \Ch$ is homologically effective and thus, by Lemma~\ref{lem:product}, so is each summand in~\eqref{eq:formula}:
\[C_* (\stdsimp{k}, \bo \stdsimp{k}) \otimes C_*(X(i_0))  \otimes \mathbb{Z}\icat (-, i_k) \in \ChI\]
The direct sum is then homologically effective too, making $G_k (C_* (X^\cofr))$ itself homologically effective. Now we can apply Lemma~\ref{lem:main} to complete the proof.
\end{proof}

In order to construct Postnikov invariants in the Postnikov tower of a diagram $\dY$, the following proposition will be used. Before the statement itself, we define the diagram of cycles $\dZ$: Given an effective diagram of chain complexes $\dC\in \ChI$, there is a diagram of cycles $\dZ_\thedimm \in \AbI$ such that $\dZ_\thedimm (i)$ is the subgroup of cycles in $\dC_\thedimm(i)$.

\begin{prop}\label{prop:gen}
Let $\dC \in \ChI$ be an effective diagram of chain complexes such that $\dH_\thedimm (\dC) = 0$ for $\thedimm \leq n$. Then there is a (computable) retraction $r \colon \dC_{\thedim+1} \to \dZ_{\thedim+1}$ i.e.\ a homomorphism that restricts to the identity on $\dZ_{\thedim+1}$.
\end{prop}
\begin{proof}
The proof is a straightforward generalization of the proof of Proposition 2.12 in~\cite{aslep}. We will compute inductively a contraction $\s \colon \dC_{\thedimm} \to \dC_{\thedimm+1}$, for $\thedimm \leq \thedim$, i.e.\ a map satisfying $\bo \s + \s \bo = \id$ and we use it to split off the cycles, namely we set $r = \id - \s \bo$.

Since $C$ is effective, the cells $e_\alpha$ form a set of free generators of $C_k \in \AbI$. Thus, we only need to compute $\s(e_\alpha)$ so that
\[\bo \s (e_\alpha) + \s \bo (e_\alpha) = e_\alpha\]
i.e.\ $\bo \s (e_\alpha) = e_\alpha - \s \bo (e_\alpha)$. This $\s(e_\alpha)$ is computed by a Smith normal form algorithm (or the more general Lemma~\ref{l:preimage}), provided that it exists; since $B_k(C) = Z_k(C)$ by assumption, we only need to verify the following (by a very easy argument using induction hypothesis)
\[\bo(e_\alpha - \s \bo (e_\alpha)) = 0.\qedhere\]
\end{proof}

\subsection{Representing a map of diagrams by an effective cocycle}
In the Postnikov system algorithm, we will encounter the
following situation: We consider a homologically effective diagram $\dX \in \sSetI$, so that there is given a strong equivalence $\dC_*(\dX)$ \mbox{$\steq$} $\dC^{\ef}_*(\dX)$ to an effective diagram $\dC^{\ef}_*(\dX)$. Let  $f \colon \dC_*(\dX) \to \dC^{\ef}_*(\dX)$ be the composite (natural) map in the strong equivalence.

Let us also consider a $(\thedimm+1)$-cocycle 
\[
\psi^{\ef}\in Z^{\thedimm+1}(\Hom (\dC^{\ef} _*  (\dX), \dpi)) = {Z^{\thedimm+1}_{\ef}}(\dX; \dpi)\]
for some diagram of effective Abelian groups $\dpi$. The superscript ``ef'' emphasise that the cocycle belongs
to the ``effective'' cochain complex ${C^{*} _{\ef}}  (\dX; \dpi)$ obtained from the effective diagram $\dC^{\ef}_*(\dX)$ associated to~$\dX$. Then $\psi^\ef$ can be represented by
a system of finite matrices, since it can be seen as a collection of maps from chain groups $ \dC_{\thedimm+1} ^{\ef}(\dX)$ of finite rank into~$\dpi(i)$, $i\in \icat$.

The composition $\psi=f \psi^\ef \colon C_*(X) \to C^{\ef}_*(X) \to \pi$ is then also a cocycle and thus corresponds to a simplicial map $\widehat\psi \colon \dX\to \overline WK(\dpi,\thedimm)$.

In the construction of Postnikov systems, we will encounter the following situation: We are given
a diagram of simplicial sets $\dP\in \Diag$, plus a mapping
$f \colon \dP\to \overline WK(\dpi,\thedim)$, for some diagram of Abelian groups $\dpi \in \AbI$
and a fixed~$\thedim\ge 1$. Now we define a diagram $\dQ\in \Diag$ as the \emph{pullback}
according to the following commutative diagram:
\[
\xymatrix{
\dQ \ar[r] \ar[d] &  WK(\dpi,\thedim) \ar[d]^{\cobo}\\
\dP\ar[r]^-{f} & \overline WK(\dpi,\thedim)
}
\]

A result from~\cite{polypost}, then gives the following
\begin{corollary}[Corollary~3.18 in \cite{polypost}]\label{c:pullback} Given $\dpi, \thedim, \dP, f$ as above, where $\dpi$ is an effective diagram of Abelian groups, the diagram $\dP$ is pointwise homologically effective, and $f$ is computable, the pullback diagram $\dQ$ is pointwise homologically effective.
\end{corollary}

\subsection{Computational complexity}
As explained, for the running time analysis we need the polynomial time versions of the above results, i.e\ we view them as suitable polynomial time constructions. The majority of these claims are straightforward generalizations of the results given in \cite[Section 3]{polypost}.
\begin{itemize}
\item Perturbations lemmas: The polynomial time version of Lemma~\ref{lem:bpl}, requires a stronger nipotency condition: for every $k \geq 0$, there exist some $N(k) \geq 0$ such that $(\eta\delta)^{N(k)}(x) = 0$ for all $x\in C_k$. Under this condition, if the input data (chain complex and reduction) are polynomial time, then so is the output chain complex and reduction data. This further implies a polynomial time version of Lemma~\ref{lem:main}.
\item Polynomial time versions of Lemma~\ref{lem:product} and Corollary~\ref{c:pullback} are obtained in the same way as in ~\cite[Section 3]{polypost}, from which Proposition~\ref{prop:cofibrep} follows.
\item A polynomial time version of Proposition~\ref{prop:gen} follows from the fact that there is a polynomial time algorithm computing the Smith normal form.
\end{itemize}

\section{Postnikov tower for diagrams}\label{sec:postnikov-proof}
In this section, we formally define an algorithmic (homologically effective) version of the Postnikov system of a diagram $Y$ which is used for computations in this paper. Then we describe an algorithm that produces the algorithmic version of the tower in case $Y$ is pointwise homologically effective and $1$-connected. Formally, we first state the existence of such algorihm in Theorem~\ref{t:postnikov-restat}. The proof is then given by first describing the algorithm and then proving its correctness. The algorithm was originally presented in first author's thesis~\cite{Filakovsky2016thesis}. Here, we give a slightly shorter version that covers the most important points of the construction.


\begin{definition}
Let $Y \in \Diag$ be a pointwise homologically effective diagram, $\thedim \in \bbN$. We say that the $n$-stage Postnikov system (tower) for $Y$ is \emph{pointwise homologically effective} if the following is provided:
\begin{itemize}
\item Pointwise homologically effective diagrams $\Pst{0}, \Pst{1},\ldots,\Pnewst \in \Diag$.
\item Effective diagrams of Abelian groups $\dpi_\thedimmm(\dY)$  representing the homotopy groups of $Y$, $1\leq \thedimmm\le\thedim$.
\item Computable maps $\varphist{\thedimmm} \colon \Yst{\thedimmm} \to \Pst{\thedimmm}$,
$1\leq\thedimmm \le\thedim$.
\item
Computable maps representing Postnikov classes $\kkk_{\thedimmm} \colon \Pst{\thedimmm-1}^\cofr \to \overline WK(\dpi_\thedimmm(\dY),\thedimmm)$,
$1<\thedimmm\le\thedim$.
\end{itemize}
\end{definition}

\begin{thm}[Precise formulation]\label{t:postnikov-restat}
Let $\thedim\geq 2$ be fixed, let $\dY \colon \icat \to \sSet$ be a finite pointwise homologically effective diagram such that every space in the diagram $\dY$ is
$1$-connected. Then there is an algorithm that computes the pointwise homologically effective $n$-stage Postnikov system for $\dY$.
\end{thm}

\subsection{Description of the algorithm}

The algorithm we present here is in fact a modification of an algorithm that constructs a Postnikov tower for 1--connected simplicial sets presented in \cite{polypost}. The main difference can be seen in the application of Proposition~\ref{prop:gen}, which will be stressed later.

The following is a pseudo-code for the algorithm in Theorem~\ref{t:postnikov-restat}:

\begin{enumerate}
[topsep=2pt,labelsep=1em,labelindent=0.5em,leftmargin=*,label=\textbf{(\arabic*)},align=left]
\item\label{bas:s} Set $\Yst{0} = \dY$, set $\Pst{0} = \{ * \}$  and construct the (obvious) map $\varphist{0}\colon \Yst{0} \to\Pst{0}$.
\item For $\thedimmm = 1$ to $\thedim$ do:
\item\label{iter:f} Compute the cofibrant replacement of $\varphist{\thedimmm -1}$ using Proposition~\ref{prop:cofibrep}. We thus obtain
\[
\varphist{\thedimmm -1}^\cofr\colon \Yst{\thedimmm}  = \Yst{\thedimmm -1} ^\cofr \longrightarrow \Pst{\thedimmm -1}^\cofr.
\]

\item\label{iter:s}
Construct the homologically effective mapping cone $\dM\coloneqq\dCone(\varphist{\thedimmm-1}^\cofr)$, with a strong equivalence $\dM \steq \dM^\ef$ to an effective diagram $\dM^\ef$.
\item\label{step:retraction} Compute a retraction $r\colon \dM^\ef _{\thedimmm+1} \to \dZ_{\thedimmm+1}  (\dM^\ef)$ using Proposition~\ref{prop:gen}. 
\item Compute the homology group $\dH_{\thedimmm+1}(\dM^\ef)$ and the composite morphism \[\rho\colon\dM^\ef _{\thedimmm+1}\stackrel{r}{\longrightarrow}\dZ_{\thedimmm+1}  (\dM^\ef) \to \dH_{\thedimmm+1}(\dM^\ef).\]
\item Set $\dpi_\thedimmm \coloneqq \dH_{\thedimmm+1}(\dM^\ef)$.
\item Denoting by $f\colon{\dM}_{\thedimmm+1} \to {\dM}_{\thedimmm+1}^\ef$ the composite chain homomorphism in the given strong equivalence, consider the composition
\[\dC_\thedimmm(\Yst{\thedimmm}) \oplus \dC_{\thedimmm+1} ({\Pst{\thedimmm-1}^\cofr}) = \dM_{\thedimmm+1} \xlra{f} \dM^\ef _{\thedimmm+1}\stackrel{\rho}{\longrightarrow}\pi_\thedimmm\]
This yields, by restriction, a cochain $\lambda_\thedimmm\colon\dC_\thedimmm(\Yst{\thedimmm})\to \dpi_\thedimmm$. Compute the simplicial map $\ell_\thedimmm:\Yst{\thedimmm} \to WK(\dpi_\thedimmm,\thedimmm)$  corresponding to $\lambda_\thedimmm$ using Proposition~\ref{prop:EML-map}.

The other restriction $\kappa_{\thedimmm} \colon \dC_{\thedimmm+1} ({\Pst{\thedimmm-1}^\cofr})\to \dpi_{\thedimmm}$ is a cocycle. Compute the corresponding simplicial map $\kkk_{\thedimmm}\colon\Pst{\thedimmm-1}^\cofr \to \overline WK(\dpi_\thedimmm,\thedimmm)$, again via Proposition~\ref{prop:EML-map}.

\item\label{step:pullback}Apply Corollary~\ref{c:pullback} to obtain $\Pst{\thedimmm}$ as a pullback in the diagram
\begin{equation}\label{eq:towerpullback}
\xymatrix@C=40pt{
{} & \Pst{\thedimmm} \ar@{->>}[d] \ar[r] & WK(\dpi_\thedimmm,\thedimmm) \ar@{->>}[d]^{\cobo} \\
\Yst{\thedimmm} \ar[r]^{\varphist{\thedimmm-1}^\cofr}  \ar@/^45pt/[urr]^{\ell_{\thedimmm}} \ar[ur]^{\varphist{\thedimmm}}& \Pst{\thedimmm-1}^\cofr \ar[r]^-{\kkk_{\thedimmm}} & \overline WK(\dpi_\thedimmm,\thedimmm)
} 
\end{equation}
and set $\varphist{\thedimmm} = (\varphist{\thedimmm -1}^\cofr,\ell_\thedimmm)$ as the map to the pullback $\Pst{\thedimmm}$.
\end{enumerate}
\subsection{Correctness of the algorithm}\label{sec:correctness}
The correctness of the algorithm follows nearly directly from the proof of Theorem~4.1 in~\cite{polypost}, where this is proven ``pointwise''.  Here, one has to show further that maps $r,\rho, \lambda_\thedimmm, \ell_\thedimmm, \kappa_{\thedimmm}, k_{\thedimmm}$ are well defined morphisms of diagrams, which is a matter of technical verification and is described in full in~\cite{Filakovsky2016thesis}. Further, we use the fact that ${\Pst{\thedimmm-1}^\cofr}$ is homologically effective, hence so is the mapping cone $M$ and Proposition~\ref{prop:gen} may be applied.

\subsection{Computational complexity}
Considering the computational complexity, the algorithm described in Theorem~\ref{t:postnikov-restat}, can be formulated as follows:

\begin{theorem}\label{t:postnkov_poly}
Let $n$ and $\mcI$ be fixed. The association $Y \mapsto P$, that takes a 1-connected diagram $Y$ of finite simplicial sets and gives its pointwise homologically effective $n$-stage Postnikov tower, is a polynomial time construction.
\end{theorem}

\begin{proof}
We split the algorithm of Theorem~\ref{t:postnikov-restat} into two inductive claims:
\begin{enumerate}
\item\label{step:cofibr-rep}Given a polynomial time computable map $\varphist{\thedimmm - 1} \colon \Yst{\thedimmm -1} \to \Pst{\thedimmm -1}$ between polynomial-time pointwise homologically effective diagrams, its cofibrant replacement $\varphist{\thedimmm - 1}^\cofr$ is a polynomial time map between polynomial time homologically effective diagrams $\Yst{\thedimmm}$, $\Pst{\thedimmm - 1}^\cofr$. 
\item Given a polynomial time computable map $\varphist{\thedimmm-1}^\cofr \colon \Yst{\thedimmm} \to \Pst{\thedimmm - 1}^{\cofr}$ between poly\-nomial-time homologically effective diagrams, the remaining data in the diagram~\eqref{eq:towerpullback} consists of polynomial time pointwise homologically effective diagrams and polynomial time computable maps.
\end{enumerate}
First point follows from the proof of Proposition~\ref{prop:cofibrep}. The second point is achieved ``pointwise'' using methods from~\cite{polypost}, thus they are polynomial time. The only difference is in the computation of retraction  $r\colon \dM^\ef _{\thedimmm+1} \to \dZ_{\thedimmm+1}  (\dM^\ef)$ which is polynomial time by Lemma~\ref{prop:gen}.
\end{proof}
\section{Applications}\label{sec:applications}


\subsection{The Tverberg-type problem}\label{s:tverberg}
In the article~\cite{MabillardWagner:Elim_II_SoCG-2016}, Mabillard and Wagner formulated the following generalization of the classical Haefliger-Weber theorem:
\begin{theorem}[Theorem 1 in \cite{MabillardWagner:Elim_II_SoCG-2016}]
Let $K$ be a (finite) simplicial complex of dimension $k$ and $r,d \in \mathbb{N}$ such that $r\geq 2, d - k \geq 3$ and $rd \geq (r+1)k +3$. Then there is an $r$-almost embedding $f\colon K \to \bbR^d$ if and only if there exists a $\mathbb{S}_r$-equivariant map $K^{r} \setminus \Delta_r \rightarrow S^{d(r-1)-1}$.
\end{theorem}
Here the space $K^{r} \setminus \Delta_r $ is the $r$-fold product of $K$, with the ``fat'' diagonals removed. Its cells can be viewed as $r$-tuples $(\sigma_1, \ldots \sigma_r)$, where $\sigma_i \in K$ and $\sigma_i \cap \sigma_j = \emptyset$ for $i\neq j$. The action of the symmetric group ${\mathbb{S}_r}$ on $K^{r} \setminus \Delta_r$ is induced by the permutation action on $K^{r}$ and is thus free. 

The sphere $S^{d(r-1)-1}$ is homotopy equivalent to $(\bbR^d)^r \setminus \delta_r$, for $\delta_r = \{(y,\cdots, y)|\, y\in \bbR^d\}$, i.e.\ the $r$-fold products with the ``thin'' diagonal removed and the action on $S^{d(r-1)-1}$ is induced by the action on $(\bbR^d)^r$, which permutes the $r$-factors. It follows that the action on $\mathbb{S}_r$, has fixed points and for any $H\leq \mathbb{S}_r$, we get $((\bbR^d)^r \setminus \delta_r)^H \sim (S^{d(r-1)-1})^H = S^{dq-1}$, where $0\leq q \leq (r-1)$. We remark that given $K \leq \mathbb{S}_r$ subconjugate to $H$, we get $(S^{d(r-1)-1})^H \subseteq (S^{d(r-1)-1})^K$.
\begin{remark}
We remark that the proof of Mabillard-Wagner theorem in~\cite{mabillard_tverberg_proof} was criticised by  A.~Skopenkov, who summarized his critique in~\cite{Skopenkov17-1} and proved the theorem in~\cite[Theorem~1.2]{Skopenkov17-2}.
\end{remark}
\begin{proof}[Proof of Theorem~\ref{thm:tverberg}]
The proof is a consequence of Theorem~\ref{cor:equiv_main} and Remark~\ref{rem:subcategory}. We notice that $\conn ((S^{d(r-1)-1})^H) = \conn(S^{dq-1}) = dq-2$, $0\leq q \leq (r-1)$, i.e. in some cases the connectivity can be less than one. From conditions $r\geq 2, d - k \geq 3$ and $rd \geq (r+1)k +3$, we can see that if for some $H\subset \mathbb{S}_r$,  $(S^{d(r-1)-1})^H \neq \emptyset$, then $\conn (S^{d(r-1)-1})^H \geq 1$. 

Let $\jcat$ be the full subcategory of $\OG$ on the objects $G/H$ where $(S^{d(r-1)-1})^H$ is nonempty. By Remark~\ref{rem:subcategory}, we get
\[
[K^{r} \setminus \Delta_r ,S^{d(r-1)-1}]_{\sSetOG} \cong [K^{r} \setminus \Delta_r ,S^{d(r-1)-1}]_{\sSetJ}.
\]
Conditions $\dim (K^{r} \setminus \Delta_r)(j)\leq 2\conn (S^{d(r-1)-1})(j)$ and $\conn (S^{d(r-1)-1})(j) \geq 1$ are  satisfied for all objects $j\in  \jcat$, so it remains to check that $rd \geq (r+1)k +3$, $r >2, d - k \geq 3, $ implies $rk \leq 2 (d(r-1)-2)$. The application of Theorem~\ref{thm:main_formulation} gives us the result.
\end{proof}
\subsection{Equivariant stable homotopy groups of spheres}
In this section, we describe how Theorem~\ref{cor:equiv_main} can be applied to the computation of equivariant stable homotopy groups of $G$-spaces (represented as $G$-simplicial sets). We showcase this on the example of $\bbZ_2$ equivariant stable homotopy groups of spheres. We are using~\cite[Chapter IX]{may1996equivariant} as our source of definitions in this section.

\heading{Basic notions}
Let $G$ be a finite group, and let $V$ be a \emph{representation of }$G$, i.e.\ a real inner product space (in our case we always assume a finite dimensional vector space over~$\mathbb{R}$) on which $G$ acts via linear isometries.\footnote{One can see the representation also as a homomorphism $G \to O(V)$.}
For a representation $V$, we have the unit disc $D(V) = \{x \in V, \lVert x \rVert \leq 1\}$ and sphere $S(V) = \{x \in V, \lVert x \rVert \leq 1\}$ and finally $S^V$ - a one point compactification (or $S^V \cong D(V)/S(V)$).

For a based $G$-space $X$, we write $\Sigma^V X = X\wedge S^V$. As an application of Freudenthal suspension theorem, we obtain that the map
\[
\Sigma^V \colon [X,Y]_G \to [\Sigma^V X,\Sigma^V Y]_G
\]
is surjective if
\begin{enumerate}[label=(\roman*)]
\item\label{freud:1} $\dim (X^H) \leq 2\conn (Y^H) + 1$ for all subgroups $H$ such that $V^H \neq 0$.
\item\label{freud:2} $\dim (X^H)\leq \conn(Y^K)$ for all pairs of subgroups $K\leq H$ with $V^K\neq V^H$.
\end{enumerate}
and bijective if the inequality is strict.

Let us describe how one can see $S^{V_0}$, where $V_0$ is the regular representation of a finite group $G$ (over $\mathbb{R}$): The regular representation is a vector space of dimension $|G|$ and, identifying the coordinate unit vectors with the elements of $G$, the group $G$ acts by swapping them. Clearly, the dimension of $S^{V_0}$ is $|G|$. From here one can deduce a simplicial set model of $S^{V_0}$. For example, for $G = \mathbb{Z}_2$, the space $S^{V_0}$ is just a 2-sphere and we can model it by glueing two discs $D_+, D_-$ along their boundary (the equator), group acts by switching the discs, while it keeps the equator fixed, thus $(S^{V_0})^{ \mathbb{Z}_2}\cong S^1$. In simplicial sets, we can thus model $S^{V_0}$ by having two $2$-simplices $\sigma, \tau$ and glue them along their faces, say by setting $d_0 \sigma = d_0 \tau$ and $d_1 \sigma = d_1 \tau = d_2 \sigma  = d_2 \tau = s_0 d_0 d_0 \sigma$ is the basepoint. The action of $\mathbb{Z}_2$  swaps $\sigma$ and $\tau$.

\begin{Def}[IX, Definition 2.1, \cite{may1996equivariant}]
A $G$-universe $U$ is a countable direct sum of representations such that $U$ contains a trivial representation and also contains each of its sub-representations infinitely often. Thus we can write $U$ as a direct sum of subspaces $(V_i)^\infty$, where $\{V_i\}$ runs through the set of distinct irreducible representations of $G$.
Universe is \emph{complete} if, up to isomorphism, it contains every irreducible representation of $G$. If $G$ is finite and $V$ is its regular representation, then $U  = V^\infty$ is a complete $G$-universe. A finite dimensional sub-$G$ space of $U$ is said to be an \emph{indexing space} in $U$.
\end{Def}

We define the equivariant stable homotopy classes of maps $X \to Y$ as
\[
\{X,Y\}_G = \colim_V\, [\Sigma^V X,\Sigma^V Y]_G
\]
Where $V$ goes through the indexing spaces in $U$ and the colimit is taken over functions 
\[ (-\wedge S^{W-V}) \colon [\Sigma^V X,\Sigma^V Y]_G \to [\Sigma^W X,\Sigma^W Y]_G ; V\subset W
\]
that are given by sending a map $\Sigma^V X \to \Sigma^V Y$ to its smash product with the identity of $S^{W-V}$.
\begin{corollary}[IX, Corollary 2.3, \cite{may1996equivariant}]\label{cor:equivstable}
If $G$ is finite and $X$ is finite dimensional, the Frudenthal suspension theorem implies the existence of a finite dimensional representation $V_0 = V_0 (X)$ such that for any representation $V$
\[
 \Sigma^{V} \colon [\Sigma^{V_0} X,\Sigma^{V_0} Y]_G \to [\Sigma^{V_0 \oplus V} X,\Sigma^{V_0 \oplus V} Y]_G 
\]
is an isomorphism.

The definition of $\{X,Y\}_G$ then gives us 
\[
\{X,Y\}_G = [\Sigma^{V_0} X,\Sigma^{V_0} Y]_G.
\]
\end{corollary}
\begin{proof}
The main idea is to choose a finite dimensional $G$-representation $U$ satisfying $\dim (U^K) >\dim (U^H)>0$ for any pair of subgroups $K \leq H \leq G$. Clearly,  the regular representation, seen as the $\mathbb{R}[G]$-module $\mathbb{R}[G]$, is such a representation. It is enough to observe that for an element $x =\sum_{k\in K} k$ , it is true that $xk =x$, but for any $h\in H\setminus K$, we get $xh \neq x$. 

From a simple dimension and connectivity comparison it follows that there exists an integer $k$ such that for $V_0 = kU = \underbrace{U \oplus \cdots \oplus U}_{k \text{-times}}$, we get
\begin{itemize}
\item $\dim ((\Sigma^{V_0} X)^H) < 2\conn((\Sigma^{V_0} Y)^H)+1$;
\item $\dim ((\Sigma^{V_0} X)^H)  < \conn((\Sigma^{V_0} Y)^K)$ for all pairs of subgroups $K\leq H$.
\end{itemize}
The inequalities  \ref{freud:1}, \ref{freud:2} then imply the result.
\end{proof}
From the corollary above and the definition of $\{X,Y\}_G$, we get that for finite simplicial sets  with an action of a finite group $G$ there exists a (finite dimensional) $G$-representation $V$ such that
\[
[\Sigma^{kV} X,\Sigma^{kV} Y]_G = \{X,Y\}_G.
\]

\begin{example}
Let us detail this general procedure in the case of the stable $\mathbb{Z}_2$-equivariant homotopy groups: from the discussion above, $\{S^n, S^0\}_{\mathbb{Z}_2}$ is isomorphic to $[\Sigma^{kV} S^n, \Sigma^{kV} S^0]_{\mathbb{Z}_2}$, where $V$ is the regular representation of $\Z_2$. In this case, we have $S^V = S^2$ and $(S^V)^{\mathbb{Z}_2} = S^1$ and we can model this space as a simplicial set with two $2$-cells, one $1$-cell and one $0$-cell.

Thus $\dim(\Sigma^{kV}S^n) = 2k +n $, $\conn \Sigma^{kV}S^n = 2k +n -1$, $\dim((\Sigma^{kV}S^n)^{\mathbb{Z}_2} = k + n$  and  $\conn (\Sigma^{kV}S^n)^{\mathbb{Z}_2} = k+n-1$. As we further have $\Sigma^{kV}S^0 = \Sigma^{kV}$, we conclude that we are in the stable range if $n < k-1$, thus to compute the $n$-th stable homotopy group of $S^0$, we should use $n+1$ fold suspension with the regular representation.

\end{example}

\heading{Proof of Theorem~\ref{thm:equivstable1}}
According to Corollary~\ref{cor:equivstable}, $\{X,Y\} = [\Sigma^{V_0} X,\Sigma^{V_0} Y]_G$, where $V_0 = \underbrace{V \oplus \cdots \oplus V}_{k \text{-times}}$ is a sum of $k$ regular representations of~$G$.

Elmendorf's theorem then implies that 
\[
[\Sigma^{V_0} X,\Sigma^{V_0} Y]_G \cong [\Phi \Sigma^{V_0} X,\Phi \Sigma^{V_0} Y]_G
\]
where the latter set is computable by our main result.

It remains to compute $k$ which depends only on $\dim \Phi X$ and $\conn \Phi Y$ -- we pick the smallest $k\in \mathbb{N}$ such that formulas \ref{freud:1} and \ref{freud:2} are satisfied.

\qed

\vfill
\vbox{\footnotesize%
\noindent\begin{minipage}[t]{0.45 \textwidth}
\raggedright
{\scshape Luk\'a\v{s} Vok\v{r}\'inek}\\
Department of Mathematics and Statistics,\\
Masaryk University,\\
Kotl\'a\v{r}sk\'a~2, 611~37~Brno,\\
Czech Republic
\end{minipage}
\hfill
\noindent\begin{minipage}[t]{0.45 \textwidth}
\raggedright
{\scshape Marek Filakovsk\'{y}}\\
Department of Algebra,\\
Faculty of Mathematics and Physics\\
Charles University,\\
Sokolovsk\'{a} 49/83, 186~00~Prague,\\
Czech Republic
\end{minipage}
}

\end{document}